\documentclass[11pt,reqno]{amsart}
\usepackage{a4wide}

\usepackage{amsthm, amssymb, amstext, amscd, amsfonts, amsxtra, latexsym, amsmath, comment, tikz,
	mathrsfs, mathtools, esint, stmaryrd, cancel, multicol}
\usepackage{color}
\usepackage{fancybox}
\usepackage{fullpage}
\usepackage[english]{babel}
\usepackage[latin1]{inputenc}
\usepackage{enumitem}
\numberwithin{equation}{section}
\newtheorem{theorem}{Theorem}
\numberwithin{theorem}{section}

\newtheorem{proposition}[theorem]{Proposition}

\newtheorem*{conjecture}{Conjecture}
\theoremstyle{remark}
\newtheorem*{remark}{Remark}

\theoremstyle{definition}

\newcommand{\Z}{\mathbb{Z}}
\newcommand{\N}{\mathbb{N}}

\newcommand{\R}{\mathbb{R}}

\newcommand{\z}{\mathfrak{z}}

\makeatletter
\newcommand{\vast}{\bBigg@{4}}
\newcommand{\Vast}{\bBigg@{5}}
\makeatother

\renewcommand{\b}[1]{\boldsymbol{#1}}

\renewenvironment{proof}[1][Proof]{\begin{trivlist} \item[\hskip \labelsep {\bfseries #1:}]}{\qed\end{trivlist}}

\author[K. Bringmann]{Kathrin Bringmann}
\author[C. Jennings-Shaffer]{Chris  Jennings-Shaffer}
\author[K. Mahlburg]{Karl Mahlburg}

\address{Mathematical Institute, University of Cologne, Weyertal 86-90, 50931 Cologne, Germany}
\email{kbringma@math.uni-koeln.de}

\address{Mathematical Institute, University of Cologne, Weyertal 86-90, 50931 Cologne, Germany}
\email{cjenning@math.uni-koeln.de}

\address{Department of Mathematics, Louisiana State University, Baton Rouge, LA 70803, USA}
\email{mahlburg@math.lsu.edu}

\title{Proofs and reductions of various conjectured partition identities of Kanade and Russell}
\allowdisplaybreaks
\begin{document}
\begin{abstract}
We prove seven of the Rogers-Ramanujan type identities  modulo $12$ that were conjectured by Kanade and Russell. Included among these seven are the two original modulo $12$ identities, in which the products have asymmetric congruence conditions, as well as the
three symmetric identities related to the principally specialized characters of certain level $2$ modules of $A_9^{(2)}$. We also give reductions of
four other conjectures in terms of single-sum basic hypergeometric
series.
\end{abstract}

\allowdisplaybreaks
\maketitle

\section{Introduction and Statement of Results}

The study of so-called ``sum-product'' identities for hypergeometric $q$-series has a long and rich history, with deep connections to the theory of integer partitions, modular forms, and affine Lie algebras. One of the earliest and most notable examples are the Rogers-Ramanujan identities \cite{RR19}, which are written in an analytic form as
 \begin{equation}
\label{E:RR}
\sum_{n \geq 0} \frac{q^{n^2}}{(q;q)_n} = \frac{1}{\left(q,q^4; q^5\right)_\infty}, \qquad \qquad
\sum_{n \geq 0} \frac{q^{n^2+n}}{(q;q)_n} = \frac{1}{\left(q^2,q^3; q^5\right)_\infty}.
\end{equation}
Here we define for $n\in\mathbb{N}_0\cup\{\infty\}$, $m\in\N_0$, and $x,q\in\mathbb{C}$ with $|q|<1$,
\begin{align*}
(x;q)_n &:= \prod_{i=0}^{n-1}\left(1-xq^i\right)
,&
(x_1,x_2,\dotsc,x_m;q)_n &:= (x_1;q)_n(x_2;q)_n\dotsb(x_m;q)_n.
\end{align*}

The identities in \eqref{E:RR} may also be interpreted combinatorially as identities between the enumeration functions for integer partitions. For example, the first identity in \eqref{E:RR} equivalently states that the number of partitions of $n$ where
successive parts differ by at least $2$ is the same as the number of
partitions of $n$ where each part is congruent to $1$ or $4$ modulo
$5$. As such, identities of this shape are also called ``gap-product'' identities, where the sum side enumerates partitions which satisfy certain restrictions on the differences between parts, and the product side enumerates partitions whose parts are restricted to lie in certain residue classes.

The Rogers-Ramanujan identities play a major role in the theory of hypergeometric $q$-series, as well as the combinatorial and analytic theory of partitions, as a large number of deep techniques have been developed in order to prove a vast collection of generalized sum-product identities. This includes the direct generalizations of \eqref{E:RR} due to Andrews \cite{Andrews3}, Bressoud \cite{Bre}, and Gordon \cite{Gordon1}, as well as Slater's lengthy lists of similar identities in \cite{Slater1,Slater2}.

Inspired by such results, as well as similar identities mentioned below (notably Capparelli's work in \cite{Capparelli1,Capparelli2}), Kanade and Russell conducted
an extensive search for new gap-product identities in \cite{KR14}. Their method was to explicitly
construct partitions satisfying three types of conditions. The first condition
being on the smallest part, both the smallest part allowed and
how many times this allowed part can appear. The second
being a difference at a distance condition,
meaning a requirement that the
difference between the parts $\pi_i$ and $\pi_{i+k}$ is at least $d$,
for fixed $k$ and $d.$
The third being a congruence at a distance condition,
meaning that if the difference
between parts $\pi_i$ and $\pi_{i+k}$ is at most $d$,
then the sum of successive parts
$\pi_i+\pi_{i+1}+\dotsb +\pi_{i+k}$ is congruent to $r$ modulo $m$,
for fixed $k$, $d$, $r$, and $m$. Kanade and Russell then calculated all such partitions over a wide range of possible search parameters, and used Euler's algorithm to determine when the resulting series is equivalent to a simple infinite product. In the end, Kanade and Russell found a total of six conjectural identities in \cite{KR14}; below we state the two conjectures with restrictions modulo $12$ (the remaining four conjectures involved the modulus $9$).

An entirely different approach was introduced by Lepowsky and Wilson's seminal work  \cite{LW78}, where they introduced vertex operators as a method for explicitly constructing affine Lie algebras, as well as calculating the standard modules of such algebras. This construction was generalized to $Z$-algebras \cite{LW81}. We briefly recall some basic definitions from the theory of affine Lie algebras (see \cite{Kac} for more details and standard notation, some of which we use below). If $\mathfrak{g}$ is an affine Lie algebra, and $\lambda$ is a dominant integral weight for $\mathfrak{g}$, then there exists a unique irreducible, integral, highest weight module $L(\lambda)$.
Remarkably, sum-product identities such as \eqref{E:RR} then arise by calculating the principally specialized character $\chi(L(\lambda))$ in two different ways: the product side uses the Kac-Weyl character formula and Lepowsky and Milne's ``numerator formula'' \cite{LM78}, while the partition gap conditions for the sum side are calculated using Lepowsky and Wilson's vertex operator algebra and $Z$-algebra programs for the explicit construction of highest weight modules \cite{LW78,LW82,LW84}. Furthermore, the character for $L(\lambda)$ is completely determined by the character of its associated ``vacuum space'' $\Omega(L(\lambda))$ (which consists of all of the highest weight vectors in $L(\lambda)$ for the Heisenberg subalgebra).

Indeed, Lepowsky and Wilson's construction of the standard modules of $A^{(1)}_1$ in \cite{LW84, LW85},
along with the work of Meurman and Primc \cite{MP87}, results in formulas that recover all of the generalized Rogers-Ramanujan identities in \cite{Andrews3, Bre, Gordon1}.
Lepowsky and Milne \cite{LM78} also showed that the Rogers-Ramanujan identities arise in character formulas for the level 2 standard modules for $A^{(2)}_2$, which was later proven using $Z$-algebras by Capparelli \cite{Capparelli1}.
Capparelli additionally used $Z$-algebras to construct the level 3 standard modules for $A^{(2)}_2$, which relied on the discovery of two conjectural partition identities (which were first proven by Andrews \cite{And94}, and later by Capparelli \cite{Capparelli2}; an analytic sum-side for Capparelli's identities was recently given by \cite{DL}).
These identities were a significant development in the theory of vertex operator algebras, as they were the first notable examples of sum-product identities that had not previously appeared, but were instead discovered using vertex-operator-theoretic techniques.

In \cite{KR18}, Kanade and Russell gave more identities of this flavor.
However, rather than searching based on partition gap conditions, they instead began their search from Bos' formulas for the level 2 principally specialized characters of $A_9^{(2)}$, namely
\begin{align}
\label{E:BosH1}
\chi(\Omega(L(\Lambda_0 + \Lambda_1)))
&= \frac{1}{\left(q, q^4, q^6, q^8, q^{11}; q^{12}\right)_\infty}, \\
\label{E:BosH2}
\chi(\Omega(L(\Lambda_3)))
&= \frac{\left(q^6; q^{12}\right)_\infty}{\left(q^2, q^3, q^4, q^8, q^9, q^{10}; q^{12}\right)_\infty}, \\
\label{E:BosH3}
\chi(\Omega(L(\Lambda_5)))
&= \frac{1}{\left(q^4, q^5, q^6, q^7, q^8; q^{12}\right)_\infty}.
\end{align}
Equation \eqref{E:BosH1} is Theorem 7.3 in \cite{Bos}, and \eqref{E:BosH2} and \eqref{E:BosH3} are stated in Conjecture 7.1 of \cite{Bos}. These level 2 characters of $A_9^{(2)}$ were a natural starting point to search for new identities, as the level 2 characters of $A_3^{(2)}, A_5^{(2)},$ and $A_7^{(2)}$ all correspond to known partition identities (this is further explained in \cite{Bos} and \cite{KR18}).

Kanade and Russell found three corresponding partition identities for \eqref{E:BosH1}, \eqref{E:BosH2}, and \eqref{E:BosH3}, and also discovered several additional asymmetric companions. Furthermore, they provided explicit formulas for the analytic sum sides of their conjectures, which are the generating functions for the partitions involved in the
identities. They also gave the sum sides corresponding to their original modulo $12$
conjectures, whereas in \cite{KR14} they only gave the conditions to describe
the relevant partitions. Kur{\c s}ung{\"o}z \cite{Ku18} recently
gave sum sides for their modulo $9$ conjectures and alternative
sum sides for their original modulo $12$ conjectures.

In the present paper we consider the majority of the conjectures from \cite{KR18}.
We state these in an analytic form, as that is the most convenient formulation for our later calculations. For $1 \leq \ell \leq 9$ the series are given by
\begin{equation}
\label{E:H1-9def}
H_\ell(x;q) := \sum_{i,j,k\geq 0} (-1)^k \frac{q^{(i+2j+3k)(i+2j+3k-1) + 3k^2 + A_\ell(i,j,k)}}{(q;q)_i\left(q^4;q^4\right)_j\left(q^6;q^6\right)_k}x^{i+2j+3k},
\end{equation}
where the $A_\ell(i,j,k)$ are linear polynomials given as follows,
\begin{multicols}{3}
\noindent $A_1(i,j,k) := i + 6j + 6k,$  \\
$A_4(i,j,k) := i + 3j + 3k$, \\
$A_7(i,j,k) := 2i + 4j + 6k$, \\
$A_2(i,j,k) := 2i + 2j + 6k,$  \\
$A_5(i,j,k) := 2i - j + 3k$, \\
$A_8(i,j,k) := i + j + 3k$, \\
$A_3(i,j,k) := 4i + 6j + 12k,$  \\
$A_6(i,j,k) := i$, \\
$A_9(i,j,k) := 3i + 5j + 9k$.
\end{multicols}

Furthermore, for $\ell \in \{10, 11\}$ we define
\begin{equation*}
H_{\ell}(x;q) :=
	\sum_{i,j,k\geq 0} \frac{ q^{\frac12(i+2j+3k)(i+2j+3k-1)+j^2+A_\ell(i,j,k)}}{(q;q)_i\left(q^2;q^2\right)_j\left(q^3;q^3\right)_k}x^{i+2j+3k},
\end{equation*}
where
\begin{align*}
A_{10}(i,j,k) := i + 2j + 4k,\qquad\qquad A_{11}(i,j,k) := 2i + 4j + 5k.
\end{align*}

\begin{conjecture}[Kanade-Russell \cite{KR14,KR18}]
The series are equal to the following products:
\begin{align}
\label{C:H1}
H_1(1) &= \frac{1}{\left(q,q^4,q^6,q^8,q^{11};q^{12}\right)_\infty}
,\\
\label{C:H2}
H_2(1) &= \frac{\left(q^6;q^{12}\right)_\infty}{\left(q^2,q^3,q^4;q^{6}\right)_\infty}
,\\
\label{C:H3}
H_3(1) &= \frac{1}{\left(q^4,q^5,q^6,q^7,q^8;q^{12}\right)_\infty}
,\\
\label{C:H4}
H_4(1) &= \frac{1}{\left(q;q^4\right)_\infty\left(q^4,q^{11};q^{12}\right)_\infty}
,\\
\label{C:H5}
H_5(1) &= \frac{1}{\left(q;q^4\right)_\infty\left(q^7,q^{8};q^{12}\right)_\infty}
,\\
\label{C:H6}
H_6(1) &= \frac{\left(q^3;q^{12}\right)_\infty}{\left(q,q^2;q^{4}\right)_\infty}
,\\
\label{C:H7}
H_7(1) &=\frac{\left(q^9;q^{12}\right)_\infty}{\left(q^2,q^3;q^{4}\right)_\infty}
,\\
\label{C:H8}
H_8(1) &= 	 \frac{1}{\left(q^3;q^4\right)_\infty\left(q,q^{8};q^{12}\right)_\infty}
,\\
\label{C:H9}
H_9(1) &= 	 \frac{1}{\left(q^3;q^4\right)_\infty\left(q^4,q^{5};q^{12}\right)_\infty}
,\\
\label{C:H10}
H_{10}(1) &= \frac{1}{\left(q;q^3\right)_\infty \left(q^3,q^6,q^{11};q^{12}\right)_\infty}
,\\
\label{C:H11}
H_{11}(1) &= \frac{1}{\left(q^2;q^3\right)_\infty\left(q^3,q^6,q^7;q^{12}\right)_\infty}
.
\end{align}
\end{conjecture}
We note that in some cases we slightly rewrite the product side of these conjectures.
The conjectures for $H_{10}(1)$ and $H_{11}(1)$ are the original modulo $12$
conjectures from \cite{KR14}, and the conjectures
for $H_1(1)$, $H_2(1)$, $H_3(1)$, $H_4(1)$, $H_5(1)$, $H_6(1)$, $H_7(1)$, $H_8(1)$, and $H_9(1)$
are respectively identities $1$, $2$, $3$, $4$, $4a$, $5$, $5a$, $6$, and $6a$ of \cite{KR18}.
As alluded to above, the conjectures for $H_{1}(1)$, $H_{2}(1)$, and
$H_{3}(1)$ arose from the principally specialized characters for $A_9^{(2)}$ listed in \eqref{E:BosH1}, \eqref{E:BosH2}, and \eqref{E:BosH3}. The conjectures for $H_4(1)$ through $H_9(1)$
are asymmetric companions to the conjectures for $H_1(1)$, $H_2(1)$, and $H_3(1)$.

\begin{remark}
It is not surprising that the infinite products in \eqref{C:H1}, \eqref{C:H2}, and \eqref{C:H3} are symmetric (in the sense that they consist of factors of the form $(1-q^n)^{\pm 1}$ for all $n$ in certain residue classes $\pm r$ modulo $12$. Indeed, this follows from the Lepwosky-Milne numerator formula combined with the fact that in a finite-dimensional simple lie algebra the roots always occur in symmetric pairs $(\alpha, -\alpha)$. Furthermore,
Kac and Peterson \cite{KP} showed that these products are essentially modular functions on certain congruence subgroups, which are expressible as the quotient of theta functions. However, it is striking that all of the remaining conjectures contain asymmetric products, which have occurred infrequently in the classical theory of partitions (for example, G\"ollnitz's so-called Big Theorem \cite{Gollnitz}).
\end{remark}

Kanade and Russell also gave combinatorial interpretations for the sum-sides of each of these conjectures. For example, the sum-side for $H_{1}(1)$ generates all
partitions such that if $\pi_i$ is a part then $\pi_i+1$ is not a part, odd parts do not repeat,
and if $\pi_i = \pi_{i+1}$ are even, then $\pi_i - \pi_{i-1} \geq 4$ and $\pi_{i+2}-\pi_i\ge4$.
The product side generates the partitions where each part is
congruent to $1$, $4$, $6$, $8$, or $11$ modulo $12$, and the conjectured identity is then equivalent to the statement that for all $n$ there are an equal number of partitions of both types.
The remaining conjectures are similar, and were discussed in detail
in \cite{KR18}. These interpretations are not immediately apparent from
the series \eqref{E:H1-9def}, but instead require a careful combinatorial analysis in which a partition satisfying difference conditions is decomposed into a staircase and jagged component.

In this article, we provide proofs of some of these conjectures and reductions
of others. In particular, we prove the following.
\begin{theorem}\label{Theorem1}
Conjectures \eqref{C:H1}, \eqref{C:H2}, \eqref{C:H3}, \eqref{C:H6}, \eqref{C:H7},
\eqref{C:H10}, and \eqref{C:H11} are true.
\end{theorem}

This result is the first time that any of Kanade and Russell's conjectures in \cite{KR14} and \cite{KR18} are proven. Our proofs use a variety of techniques, including summation and transformation formulas for hypergeometric $q$-series, series solutions to $q$-difference equations, and linear recurrences. The cases \eqref{C:H1}, \eqref{C:H2}, and \eqref{C:H3} are perhaps the most significant, as these partition identities involve the principally specialized characters for $A_9^{(2)}$, and are therefore likely to be useful in verifying the vertex operator construction of the corresponding standard modules (cf. \cite{Kan18}).

Although we have not yet fully proven \eqref{C:H4}, \eqref{C:H5}, \eqref{C:H8}, or \eqref{C:H9}, the following result reduces the ``sum-sides'' from \eqref{E:H1-9def} to expressions involving a single hypergeometric $q$-series in these cases.
\begin{theorem}
\label{Theorem2}
The following
identities are true:
\begin{align}
\label{T:H4}
H_4(1) &=
	\left(q^3;q^4\right)_\infty
	\sum_{n\ge0}\frac{\left(q^{3};q^6\right)_n  q^{ n^2} }
	{(q,q^2;q^2)_{n}  \left(q^{3};q^4\right)_n }
,\\
\label{T:H5}
H_5(1) &=
	\left(q^{-1};q^4\right)_\infty
	+
	\left(q^3;q^4\right)_\infty
	\sum_{n\ge0}
	\frac{\left(1+q^{2n-4}+q^{2n-1} \right)  \left(q^{-3};q^6\right)_n  q^{ n^2 + 4n+3} }
	{ \left(1-q^{2n+2}\right) \left(q^{-1},q^2;q^2\right)_{n}  \left(q^3;q^4\right)_n}
,\\
\label{T:H8}
H_8(1) &=
	\left(q;q^4\right)_\infty
	+
	\left(q^5;q^4\right)_\infty
	\sum_{n\ge0}
	\frac{\left(1+q^{2n}+q^{2n+1} \right) \left(q^{3};q^6\right)_n q^{ n^2 + 2n+1}  }
	{\left(1-q^{2n+2}\right) \left(q,q^2;q^2\right)_{n}  \left(q^5;q^4\right)_n }
,\\
\label{T:H9}
H_9(1) &=
	\left(q^5;q^4\right)_\infty
	\sum_{n\ge0}
	\frac{  \left(q^{3};q^6\right)_n  q^{ n^2+2n}}
	{\left(q,q^2;q^2\right)_{n}  \left(q^{5};q^4\right)_n }
.
\end{align}
\end{theorem}

The rest of the article is organized as follows. We begin in Section \ref{S:Prel} by recalling useful summation and transformations for hypergeometric $q$-series. The section continues with some general results for finding series solutions to $q$-difference equations, as well as results on linear recurrences for $q$-series. In Section \ref{S:H1} we use similar techniques to prove the conjecture for $H_1$. Next, we apply the general results to give brief proofs for $H_j \ (2 \leq j \leq 11)$ in Section \ref{S:H2-11}. In Section \ref{S:J12} we present a partial reduction for yet another of Kanade and Russell's conjectural identities. Finally, we conclude in Section \ref{S:Conc} with some additional discussion.

\section*{Acknowledgments}
The research of the first author is supported by the Alfried Krupp Prize for Young University Teachers of the Krupp foundation and the research leading to these results receives funding from the European Research Council under the European Union's Seventh Framework Programme (FP/2007-2013) / ERC Grant agreement n. 335220 - AQSER.

We thank Shashank Kanade, Jim Lepowsky, Jeremy Lovejoy, and the anonymous referee for helpful comments on a preliminary version of this paper.

\section{Preliminary Results}
\label{S:Prel}

\subsection{$q$-series transformations and Appell's Comparison Theorem}
We require several standard $q$-series identities, all of which can be
found in the appendix of \cite{GR}.
We also use the standard notation for basic hypergeometric series, namely,
\begin{align*}
{_{r+1}\phi_r}\left[\begin{matrix} a_1, a_2,\dotsc,a_{r+1}\\ b_1, b_2,\dotsc,b_{r} \end{matrix};q,x\right]
&:=
\sum_{n\ge0}\frac{\left(a_1, a_2,\dotsc,a_{r+1};q\right)_n}
{\left(b_1, b_2,\dotsc,b_r,q;q\right)_n}x^n.
\end{align*}
We begin with {\it Cauchy's $q$-Binomial Theorem} \cite[(II.3)]{GR},
\begin{align}\label{E:qbinomial}
{_{1}\phi_0}\left[a;q,x\right]
&=
\sum_{n\ge0}
\frac{(a;q)_n}{(q;q)_n}x^n
=
\frac{(ax;q)_\infty}{(x;q)_\infty}
,
\end{align}
which implies {\it Euler's identities} \cite[(II.1) and (II.2)]{GR}
\begin{equation}\label{E:euler}
{_{1}\phi_0}\left[0;q,x\right]
=
\sum_{n\geq 0} \frac{x^n}{(q;q)_n}
=\frac{1}{(x;q)_\infty}
,\quad\,\,\,\,\,
\lim_{a\rightarrow\infty}
{_{1}\phi_0}\left[a;q,\frac{x}{a}\right]
=
\sum_{n\geq 0} \frac{(-1)^n q^{\frac12 n(n-1)}}{(q;q)_n}x^n
=(x;q)_\infty
.
\end{equation}
We also need {\it Heine's transformation} \cite[(III.1)]{GR},
\begin{align}\label{E:heine}
{_2\phi_1}\left[\begin{matrix} a,b \\ c\end{matrix};q,x\right]
&=
	\frac{(b,ax;q)_\infty}{(c,x;q)_\infty}{_2\phi_1}
	\left[\begin{matrix}\frac{c}{b},x\\ax \end{matrix};q,b\right]
,
\end{align}
which implies (using the second variant of Heine's identities) the $q$-analog of Kummer's Theorem,
which is also known as the {\it Bailey-Daum summation} \cite[(II.9)]{GR},
\begin{align} \label{E:kummer}
{_2\phi_1}\left[\begin{matrix} a, b\\ \frac{aq}{b} \end{matrix};q,-\frac{q}{b}\right]
&=
\frac{\left(aq,\frac{aq^2}{b^2},q^2;q^2 \right)_\infty}{\left(-\frac{q}{b},\frac{aq}{b},q ;q\right)_\infty }.
\end{align}
Lastly, we have {\it Hall's formula} \cite[(III.10)]{GR},
\begin{align}
\label{E:3phi2a}
{_3\phi_2}\left[\begin{matrix}
	a, b, c \\
	d,e\end{matrix};q,\frac{de}{abc}\right]
&=
\frac{\left( b,\frac{de}{ab},\frac{de}{bc} ;q\right)_\infty}
{\left( d,e,\frac{de}{abc} ;q\right)_\infty}
{_3\phi_2}\left[\begin{matrix}
	\frac{d}{b}, \frac{e}{b}, \frac{de}{abc} \\[0.25ex]
	\frac{de}{ab},\frac{de}{bc}\end{matrix};q,b\right]
.
\end{align}

We also make use of a result that is sometimes referred to as Appell's Comparison Theorem,
which is common when dealing with limiting cases of functional equations and recurrences.
The following statement is a slight extension of Theorem 8.2 in \cite{Rud} to allow for complex coefficients.
\begin{proposition}\label{P:Appell}
Suppose that $F(x)=\sum_{n\ge0}\alpha_n x^n$ and $\alpha_\infty := \lim_{n \to \infty} \alpha_n$ exists
(and thus $F(x)$ has radius of convergence at least one). Then
\begin{align*}
\lim_{x\rightarrow1^-}(1-x) F(x) &= \alpha_\infty.
\end{align*}
In particular, if $G(x) = \sum_{n \geq 0 } \beta_n x^n$ has radius of convergence greater than one, then $F(x) := \frac{G(x)}{1-x}$ is such a series, and
\begin{equation*}
\lim_{x \rightarrow 1^-} (1-x) F(x) = \alpha_\infty = \sum_{n \geq 0} \beta_n = G(1).
\end{equation*}
\end{proposition}
In this article we view power series as analytic functions of complex variables. This is necessary
for our use of Appell's Comparison Theorem. Any formal rearrangements are considered to be done
within the radius of convergence, and are valid as identities between analytic functions. Using the fact that each $H_j(x)$ is the generating function for some subset of partitions into distinct parts (see \cite{KR18}), we see that it is absolutely convergent for all $x \in \mathbb{C}$.

\subsection{Solving general functional equations}

We now give the first of our general formulas for functions satisfying certain
functional equations. This formula is relevant for
$H_j$ for $2\leq j \leq 9$
\begin{proposition}\label{Prop1}
Suppose that $A(x) = \sum_{n\ge0}\alpha_n x^n$
has radius of convergence greater than one, $A(0)=1$, and $A(x)$ satisfies
\begin{align}\label{E:Arec}
A(x)
&=
	\left(1+xq^{a+2}\right)A\left(xq^2\right)
	+xq^{a+b}\left(1+xq^b\right)A\left(xq^4\right)
	+x^2q^{2a+2b+2}\left(1-xq^4\right)A\left(xq^6\right)
,
\end{align}
where $|q|<1$.
Then
\begin{align*}
A(1)
&=
	\left(q^{2b+a-2};q^4\right)_\infty
	\sum_{n\ge0}
	\frac{\left(q^{3b-6};q^6\right)_n  q^{ n^2 + (a+1)n}}
	{\left(q^{b-2},q^2;q^2\right)_{n}  \left(q^{2b+a-2};q^4\right)_n }
.
\end{align*}
Furthermore, if $a=0$, then we have
\begin{align*}
A(1)
&=
\frac{\left(q^{3b}; q^{12}\right)_\infty}{\left(q^2, q^b; q^4\right)_\infty}.
\end{align*}
\end{proposition}
\begin{proof}
Set
\begin{align*}
B(x)
&:=
	\frac{A(x)}{\left(x;q^2\right)_\infty}
,
\end{align*}
so that dividing \eqref{E:Arec}
by $(xq^4;q^2)_\infty$ yields
\begin{align*}
\left(1-x\right)&\left(1-xq^2\right)B(x)\\
&=
	\left(1-xq^2\right)\left(1+xq^{a+2}\right)B\left(xq^2\right)
	+xq^{a+b}\left(1+xq^b\right)B\left(xq^4\right)
	+x^2q^{2a+2b+2}B\left(xq^6\right).
\end{align*}
Writing
\begin{equation*}
B(x) =: \sum_{n\ge0}\beta_n x^n
\end{equation*}
we obtain, after some reordering,
\begin{multline}\label{E:betaRec}
\left(1-q^{2n}\right)\beta_n=
	\left(1 + q^2 - q^{2n} + q^{2n+a} + q^{4n+a+b-4}\right)\beta_{n-1}
	\\
-q^2\left(1+q^{2n+a-2}\right) \left(1-q^{4n+a+2b-10}\right)  \beta_{n-2}
.
\end{multline}

Set
\begin{align*}
\gamma_n &:= \frac{\beta_n}{\left(-q^{a+2};q^2\right)_n},
\qquad\qquad\qquad
C(x) := \sum_{n\ge0}\gamma_n x^n.
\end{align*}
Proposition \ref{P:Appell}
applies to $B(x)$ and $C(x)$, giving that
\begin{align}
\label{E:lim(1-x)B}
\lim_{x\rightarrow 1^-} \left(1-x\right)B\left(x\right)
&=
\beta_{\infty}
=
\left(-q^{a+2};q^2\right)_\infty \gamma_\infty
=
\left(-q^{a+2};q^2\right)_\infty
\lim_{x\rightarrow 1^-} \left(1-x\right)C\left(x\right)
.
\end{align}
We proceed by dividing \eqref{E:betaRec} by $(-q^{a+2};q^2)_{n-1}$, which gives
\begin{align*}
\left(1-q^{2n}\right)\left(1+q^{2n+a}\right)\gamma_n
=
	\left(1 + q^2 - q^{2n} + q^{2n+a} + q^{4n+a+b-4}\right)\gamma_{n-1}
	-q^2\left(1-q^{4n+a+2b-10}\right)  \gamma_{n-2}
.
\end{align*}
In terms of $C(x)$, this may be rewritten as
\begin{align}\label{E:Crec}
\left(1-x\right)\left(1-xq^2\right)C\left(x\right)
&=
	\left(1-q^a\right)\left(1-xq^2\right)C\left(xq^2\right)
	+q^a\left(1+xq^b+x^2q^{2b}\right)C\left(xq^4\right)
.
\end{align}

If $a=0$, then this gives
$$
C(x) = \frac{\left(x^3 q^{3b}; q^{12}\right)_\infty}{\left(x;q^2\right)_\infty \left(xq^b; q^4\right)_\infty},
$$
which yields, also using \eqref{E:lim(1-x)B},
\begin{align*}
A(1)
&=
\left(q^2;q^2\right)_\infty
\lim_{x\rightarrow1^-}(1-x)B(x)
=
\frac{ \left(q^2,-q^2;q^2\right)_\infty \left(q^{3b};q^{12}\right)_\infty }
{ \left(q^2;q^2\right)_\infty \left(q^b;q^4\right)_\infty }
=
\frac{\left(q^{3b}; q^{12}\right)_\infty}{\left(q^2, q^b; q^4\right)_\infty}.
\end{align*}

Set
\begin{align*}
D\left(x\right)
&:=
	\frac{\left(xq^{b-2};q^2\right)_\infty \left(x;q^2\right)_\infty}{\left(x^3q^{3b-6};q^6\right)_\infty}C\left(x\right)
.
\end{align*}
and multiply \eqref{E:Crec} by
$\displaystyle \frac{(xq^{b};q^2)_\infty (xq^4;q^2)_\infty}{(x^3q^{3b};q^6)_\infty}$
to obtain
\begin{align*}
\left( 1+xq^{b-2}+x^2q^{2b-4} \right)D\left(x\right)
&=
\left(1-q^a\right)D\left(xq^2\right) + q^aD\left(xq^4\right).
\end{align*}
Writing
\begin{equation*}
D(x) =:
\sum_{n\ge0} \delta_n x^n,
\end{equation*}
we may rewrite this as
\begin{align}\label{E:deltaRec}
\left(1-q^{2n}\right)\left(1+q^{2n+a}\right)\delta_n
&=
	-q^{b-2}\delta_{n-1}-q^{2b-4}\delta_{n-2}
.
\end{align}

Set
\begin{align*}
\varepsilon_n
&:=
	\left(-q^{a+2};q^2\right)_n\delta_n.
\end{align*}
We multiply
\eqref{E:deltaRec} by $(-q^{a+2};q^2)_{n-1}$ to obtain
\begin{align*}
\left(1-q^{2n}\right)\varepsilon_n
&=
	 -q^{b-2}\varepsilon_{n-1}-q^{2b-4}\left(1+q^{2n+a-2}\right)\varepsilon_{n-2}
.
\end{align*}
Setting
\begin{equation*}
E(x) := \sum_{n\geq 0} \varepsilon_n x^n,
\end{equation*}
we obtain after rewriting
\begin{align*}
E\left(x\right)
&=
	 \frac{\left(1-xq^{b-2}\right)\left(1-x^2q^{2b+a-2}\right)}{1-x^3q^{3b-6}}E\left(xq^2\right)
.
\end{align*}
Note that $\varepsilon_0=1$, and thus
\begin{align*}
E\left(x\right)
&=
	\frac{\left(xq^{b-2};q^2\right)_\infty \left(x^2q^{2b+a-2};q^4\right)_\infty}{\left(x^3q^{3b-6};q^6\right)_\infty}
.
\end{align*}
Using \eqref{E:euler}, we find that
\begin{align*}
E\left(x\right)
&=
	\sum_{\ell,r,m\ge0}
	\frac{ \left(-1\right)^m q^{m^2 + \left(b-3\right)m} \left(-1\right)^r x^{2r} q^{2r^2 + \left(2b+a-4\right)r} x^{3\ell} q^{3\left(b-2\right)\ell} }
	{ \left(q^2;q^2\right)_m \left(q^4;q^4\right)_r \left(q^6;q^6\right)_\ell}x^m
.
\end{align*}
As such, with $m=n-2r-3\ell$,
\begin{align*}
\varepsilon_n
&=
	\sum_{\ell,r\ge0}
	\frac{ \left(-1\right)^{n+r+\ell} q^{\left(n-2r-3\ell\right)^2 + \left(b-3\right)\left(n-2r-3\ell\right) + 2r^2 + \left(2b+a-4\right)r + 3\left(b-2\right)\ell} }
	{ \left(q^2;q^2\right)_{n-2r-3\ell} \left(q^4;q^4\right)_r \left(q^6;q^6\right)_\ell}
.
\end{align*}

Returning to $D(x)$, we have
\begin{align*}
\delta_n
&=
	\sum_{\ell,r\ge0}
	\frac{ \left(-1\right)^{n+r+\ell} q^{\left(n-2r-3\ell\right)^2 + \left(b-3\right)\left(n-2r-3\ell\right) + 2r^2 + \left(2b+a-4\right)r + 3\left(b-2\right)\ell} }
	{ \left(q^2;q^2\right)_{n-2r-3\ell} \left(q^4;q^4\right)_r \left(q^6;q^6\right)_\ell \left(-q^{a+2};q^2\right)_n}
.
\end{align*}
Making the change of variables $n\mapsto n+2r+3\ell$ yields
\begin{align}\label{E:Dsum}
D(x) =	\sum_{\ell,r,n\ge0}
	\frac{ \left(-1\right)^{n+r} q^{n^2 + \left(b-3\right)n + 2r^2 + \left(2b+a-4\right)r + 3\left(b-2\right)\ell} }
	{ \left(q^2;q^2\right)_{n} \left(q^4;q^4\right)_r \left(q^6;q^6\right)_\ell \left(-q^{a+2};q^2\right)_{n+2r+3\ell}}x^{n+2r+3\ell}
.
\end{align}
We apply \eqref{E:heine} to the sum on $n$ to give
\begin{align*}
&\sum_{n\ge0}
\frac{ \left(-1\right)^{n}  q^{n^2 + \left(b-3\right)n }}
{ \left(q^2;q^2\right)_{n} \left(-q^{a+2};q^2\right)_{n+2r+3\ell}}x^{n}
=
	\frac{1}{\left(-q^{a+2};q^2\right)_{2r+3\ell}}
	\sum_{n\ge0}
	\frac{\left(-1\right)^{n}  q^{ n^2 + \left(b-3\right)n } }
	{\left(-q^{4r+6\ell+a+2},q^2;q^2\right)_{n} }x^n
\\
&=
	\frac{1}{\left(-q^{a+2};q^2\right)_{2r+3\ell}}
	\lim_{\substack{w\rightarrow\infty\\z\rightarrow 0}}
	{_2\phi_1}\left[\begin{matrix}
	w, z\\
	-q^{4r+6\ell+a+2}\end{matrix};q^2,\frac{xq^{b-2}}{w}\right]
\\	
&=
	 \frac{\left(xq^{b-2};q^2\right)_\infty}{\left(-q^{a+2};q^2\right)_{2r+3\ell} \left(-q^{4r+6\ell+a+2};q^2\right)_{\infty}}
	\lim_{\substack{z\rightarrow 0}}
	{_2\phi_1}\left[\begin{matrix}
	- q^{4r+6\ell+a+2}z^{-1}, 0\\
	xq^{b-2}\end{matrix};q^2,z\right]
\\
&=
	 \frac{\left(xq^{b-2};q^2\right)_\infty}{\left(-q^{a+2};q^2\right)_{\infty} }
	\sum_{n\ge0}\frac{ q^{n^2+\left(a+1\right)n+4rn+6\ell n}  }
	{\left(xq^{b-2},q^2;q^2\right)_n}
.
\end{align*}
We plug this back into \eqref{E:Dsum}, and evaluate the inner
sums on $r$ and $\ell$ with \eqref{E:euler}, yielding
\begin{align*}
D\left(x\right)
&=
\frac{\left(xq^{b-2};q^2\right)_\infty}{\left(-q^{a+2};q^2\right)_\infty}
\sum_{n,\ell,r\ge0}
\frac{\left(-1\right)^{r} q^{ n^2 + \left(a+1\right)n + 2r^2 + \left(2b+a+4n-4\right)r + 3\left(b-2+2n\right)\ell }}
{\left(xq^{b-2},q^2;q^2\right)_{n} \left(q^4;q^4\right)_r \left(q^6;q^6\right)_\ell } x^{2r+3\ell}
\\
&=
\frac{\left(xq^{b-2};q^2\right)_\infty}{\left(-q^{a+2};q^2\right)_\infty}
\sum_{n\ge0}
\frac{ \left(x^2q^{4n+2b+a-2};q^4\right)_\infty  q^{ n^2 + \left(a+1\right)n} }
{\left(xq^{b-2},q^2;q^2\right)_{n}  \left(x^3q^{6n+3b-6};q^6\right)_\infty }
\\
&=
\frac{\left(xq^{b-2};q^2\right)_\infty \left(x^2q^{2b+a-2};q^4\right)_\infty }
{\left(-q^{a+2};q^2\right)_\infty \left(x^3q^{3b-6};q^6\right)_\infty}
\sum_{n\ge0}
\frac{ \left(x^3q^{3b-6};q^6\right)_n q^{ n^2 + \left(a+1\right)n}  }
{\left(xq^{b-2},q^2;q^2\right)_{n}  \left(x^2q^{2b+a-2};q^4\right)_n }
.
\end{align*}
Thus
\begin{align*}
	C(x)= \frac{\left(x^2 q^{2b+a-2};q^4\right)_\infty}{\left(-q^{a+2};q^2\right)_\infty \left(x;q^2\right)_\infty}\sum_{n\ge0}
	\frac{  \left(x^3q^{3b-6};q^6\right)_n q^{ n^2 + \left(a+1\right)n} }
	{\left(xq^{b-2},q^2;q^2\right)_{n}  \left(x^2q^{2b+a-2};q^4\right)_n }.
\end{align*}
Recalling \eqref{E:lim(1-x)B}, we therefore have
\begin{align*}
A(1)
&=
\left(q^2;q^2\right)_\infty \lim_{x\rightarrow 1^-} \left(1-x\right)B\left(x\right)
=
\left(q^2,-q^{a+2};q^2\right)_\infty \lim_{x\rightarrow 1^-} \left(1-x\right)C\left(x\right)
\\
&=
	\frac{\left(q^2;q^2\right)_\infty \left(-q^{a+2};q^2\right)_\infty \left(q^{2b+a-2};q^4\right)_\infty }
	{ \left(q^2;q^2\right)_\infty \left(-q^{a+2};q^2\right)_\infty }
	\sum_{n\ge0}
	\frac{ \left(q^{3b-6};q^6\right)_n  q^{ n^2 + \left(a+1\right)n} }
	{\left(q^{b-2},q^2;q^2\right)_{n}  \left(q^{2b+a-2};q^4\right)_n }
\\
&=
	\left(q^{2b+a-2};q^4\right)_\infty
	\sum_{n\ge0}
	\frac{ \left(q^{3b-6};q^6\right)_n q^{ n^2 + \left(a+1\right)n}  }
	{\left(q^{b-2},q^2;q^2\right)_{n}  \left(q^{2b+a-2};q^4\right)_n }
.
\end{align*}
\end{proof}
\begin{remark}
When $a=0$, the series in Proposition \ref{Prop1} sums to an
infinite product. However, this alone is not a deep result.
Indeed, by letting $\omega:=e^{\frac{2\pi i}{3}}$, and then using
\eqref{E:3phi2a} and \eqref{E:qbinomial},
we find that
\begin{align*}
\sum_{n\ge0}\frac{ \left(q^{3b-6};q^6\right)_n q^{n^2+n}}{(q^{b-2},q^2;q^2)_n(q^{2b-2};q^4)_n}
&=
\lim_{x\rightarrow\infty}
{_3\phi_2}\left[\begin{matrix}
	x, \omega q^{b-2}, \omega^2 q^{b-2} \\
	q^{b-1},-q^{b-1}\end{matrix};q^2,-\frac{q^{2}}{x}\right]
\\
&=
\frac{\left( \omega q^{b-2}, -q^2;q^2\right)_\infty}{\left( q^{b-1}, -q^{b-1};q^2\right)_\infty}
{_1\phi_0}\left[	\omega q^2 ;q^4,\omega q^{b-2} \right]
=
\frac{\left( q^{3b};q^{12}\right)_\infty }
{\left( q^2,q^{b}, q^{2b-2};q^4\right)_\infty}
.
\end{align*}
\end{remark}

We require another general formula that is relevant for
$H_{10}$ and $H_{11}$.
\begin{proposition}\label{Prop2}
Suppose that $A(x)=\sum_{n\ge0}\alpha_nx^n$ has positive radius of convergence
and $A(x)$ satisfies
\begin{align}\label{E:Prop2ARec}
A(x)
&=
	\left(1+q^a+x^2q^b+x^2q^c\right)A\left(xq^3\right)
	 -q^a\left(1+x^2q^{b+c-a-5}\right)\left(1+x^2q^{11}\right)A\left(xq^6\right)
,
\end{align}
where $a\not\in 3\mathbb{Z}$ if $a\le-6$.
Then
\begin{align*}
A(x)
&=
	\alpha_0\left(-x^2q^5;q^6\right)_\infty
	\sum_{n\ge0}\frac{ \left( q^{b-5},q^{c-5};q^6 \right)_n (-1)^nq^{5n}  }
		{\left( q^{6},q^{a+6};q^6 \right)_n}x^{2n}
	\\&\quad
	+
	\alpha_1x\left(-x^2q^5;q^6\right)_\infty
	\sum_{n\ge0}\frac{ \left( q^{b-2},q^{c-2};q^6 \right)_n (-1)^nq^{5n} }
		{\left( q^{9}, q^{a+9};q^6 \right)_n}x^{2n} 	
.
\end{align*}
\end{proposition}
\begin{proof}
Setting
\begin{align*}
B(x):= \frac{A(x)}{\left(-x^2q^5;q^6\right)_\infty}
\end{align*}
and dividing \eqref{E:Prop2ARec} by
$(-x^2q^{11};q^6)_\infty$ we obtain
\begin{align*}
\left(1+x^2q^5\right)B(x)
&=
	\left(1+q^a+x^2q^b+x^2q^c\right)B\left(xq^3\right)
	-q^a\left(1+x^2q^{b+c-a-5}\right)B\left(xq^6\right)
.
\end{align*}
Writing
\begin{equation*}
B(x)  =: \sum_{n\ge0}\beta_nx^n
\end{equation*}
yields
\begin{align*}
\beta_{n}
&=-
\frac{q^5\left(1-q^{3n+b-11}\right)\left(1-q^{3n+c-11}\right)}{\left(1-q^{3n}\right)\left(1-q^{3n+a}\right)}
\beta_{n-2}.
\end{align*}
In particular,
\begin{align*}
\beta_{2n}=
\frac{(-1)^nq^{5n}\left(q^{b-5},q^{c-5};q^6\right)_n }{\left(q^{6},q^{a+6};q^6\right)_n}\beta_0
,\qquad
\beta_{2n+1} = \frac{(-1)^nq^{5n}\left(q^{b-2},q^{c-2};q^6\right)_n }{\left(q^{9},q^{a+9};q^6\right)_n}\beta_1
.
\end{align*}
However, $\beta_0=\alpha_0$ and $\beta_1=\alpha_1$, so the result
follows.
\end{proof}

\subsection{Recurrences for $H_j$ for $1\leq j\leq 9$}\label{genrec}
Recalling \eqref{E:H1-9def}, it is natural to define generalized sums for $N \in \Z$ and $c,d \in \R$ by
\begin{equation*}
\label{E:hNcdDef}
h_{c,d,N} := \sum_{j,k\in\Z} \frac{(-1)^k q^{3k^2+(2c-1)j + 3dk}}{(q;q)_{N-2j-3k}\left(q^4;q^4\right)_j\left(q^6;q^6\right)_k}.
\end{equation*}
Note that for $\alpha\in\mathbb{C}$, we let
$(x;q)_\alpha:=\frac{(x;q)_\infty}{(xq^\alpha;q)_\infty}$. In particular,
$(q;q)_n^{-1}=0$ if $n\in-\N$, and thus $h_{c,d,N} = 0$ for $N < 0$.

In order to use Proposition \ref{Prop1}, we need three-term linear recurrences for the $h_{c,d,N}$, where $N$ varies and $c$ and $d$ are fixed. We find that all such recurrences needed for Kanade and Russell's conjectures are contained in three families.

\begin{proposition}\label{P:hNcdrec}
We have the recurrence, writing $h_N := h_{c,d,N}$,
\begin{equation}
\label{E:hcdlongrec}
\left(1 - q^{2N}\right) h_{N} = (1+q)h_{N-1} + \left(q^{2c-1} - q\right) h_{N-2} - \left(q^{2c-1} +q^{2c} + q^{2N-3+3d}\right) h_{N-3} + q^{2c} h_{N-4}.
\end{equation}

The right-hand side of \eqref {E:hcdlongrec} can be replaced by an equivalent three-term recurrence if $(c,d)$ are in the following one-parameter families.

\begin{enumerate}[leftmargin=*, label={\rm (\arabic*)}]
\item If $d = 0$, then we have
\begin{equation}
\label{E:d=0rec}
\left(1- q^{2N}\right) h_N = \left(1 + q^{2N-1}\right)h_{N-1} + \left(q^{2c-1} + q^{2N-2}\right) h_{N-2} - q^{2c-1} h_{N-3}.
\end{equation}

\item If $d = -1$, then we have
\begin{equation}
\label{E:d=-1rec}
\left(1- q^{2N}\right) h_N = \left(q + q^{2N-2}\right)h_{N-1} + \left(q^{2c-1} + q^{2N-4}\right) h_{N-2} - q^{2c} h_{N-3}.
\end{equation}

\item If $c = d+\frac32$, then we have
\begin{align}
\label{E:d+3/2rec}
&\left(1 - q^{2N}\right) h_N \\
\notag
&= \left(1 + q - q^{d+1} +  q^{2N-1+d}\right)h_{N-1} + \left(-q + q^{d+1} + q^{d+2} + q^{2N-2+2d}\right) h_{N-2} - q^{d+2} h_{N-3}.
\end{align}
\end{enumerate}
\end{proposition}

\begin{proof}[Proof]
A short calculation verifies the three basic linear relations among the $h$'s, namely
\begin{align}
\label{E:h9recsystem}
h_{c,d,N}-h_{c,d,N-1}&= q^N h_{c-1,d-1,N},\\
\notag
h_{c,d,N}-h_{c,d+2,N}&= -q^{3d+3}h_{c,d+2,N-3},\\
\notag
h_{c,d,N}-h_{c+2,d,N}&= q^{2c-1} h_{c,d,N-2}.
\end{align}
We then plug in various values to obtain the system
\begin{align*}
h_{c,d,N} & = h_{c,d,N-1} + q^N h_{c-1, d-1,N}, \\
h_{c,d,N-1} &= h_{c,d,N-2} + q^{N-1} h_{c-1, d-1, N-1}, \\
h_{c,d,N-2} &= h_{c,d,N-3} + q^{N-2} h_{c-1, d-1, N-2}, \\
h_{c,d,N-3} &=  h_{c,d,N-4} + q^{N-3} h_{c-1, d-1, N-3}, \\
h_{c-1,d-1,N} & = h_{c-1,d-1,N-1} + q^N h_{c-2, d-2, N}, \\
h_{c-1,d-1,N-2} & = h_{c-1,d-1,N-3} + q^{N-2} h_{c-2, d-2, N-2}, \\
h_{c-2,d-2, N} & = h_{c,d-2,N} + q^{2c-5} h_{c-2, d-2, N-2}, \\
h_{c,d-2,N} & =  h_{c,d,N} - q^{3d-3} h_{c, d, N-3}.
\end{align*}
We view this as a system of 8 equations in the 8 variables
$$
h_{c,d,N}, h_{c-1, d-1,N}, h_{c-1,d-1,N-1}, h_{c-1,d-1,N-2}, h_{c-1,d-1,N-3}, h_{c-2,d-2,N}, h_{c-2,d-2,N-2}, h_{c,d-2,N},
$$
with constants $h_{c,d,N-1}, h_{c,d,N-2}, h_{c,d,N-3}, h_{c,d,N-4}$. It is a brief calculation in linear algebra to solve for $h_{c,d,N}$ in terms of these constants, and the result is \eqref{E:hcdlongrec}.

We next consider the special cases listed in parts $(1)$, $(2)$, and $(3)$. In order to do so, we make use of the additional ``shift'' structure in the linear system: if the $h_N$ satisfy a certain linear equation $L(N)$, then they must also satisfy $L(N\pm 1)$. To use this, we assume that $h_N$ satisfies an equation of the form
\begin{equation}
\label{E:hgenrec}
\left(1 - q^{2N}\right) h_N = A_1(N) h_{N-1} + A_2(N) h_{N-2} -q^{a} h_{N-3}.
\end{equation}
In order to use \eqref{E:hcdlongrec}, we then apply \eqref{E:hgenrec} iteratively, to obtain
\begin{multline}\label{E:hgeniter}
\left(1 - q^{2N}\right) h_N
=
\left(A_1(N) +q^{2c-a}\left(1 - q^{2N-2}\right)\right) h_{N-1}
+ \left(A_2(N) -q^{2c-a} A_1(N-1)\right) h_{N-2}  \\
 - \left(q^a + q^{2c-a} A_2(N-1)\right) h_{N-3} + q^{2c}h_{N-4}.
\end{multline}
If we can show that \eqref{E:hgeniter} is equivalent to \eqref{E:hcdlongrec}, then we can conclude that it is also equivalent to \eqref{E:hgenrec}. This is because the shorter recurrence implies the longer recurrence, and the solution of a linear recurrence is uniquely determined by its initial conditions. In this case the initial conditions are always given by $h_{c,d,N} = 0$ for $N < 0$ and $h_{c,d,0} = 1$.

By plugging in the specific polynomials and constants from \eqref{E:d=0rec}, \eqref{E:d=-1rec}, and \eqref{E:d+3/2rec}, we find
in each case that \eqref{E:hgeniter} reduces to
\eqref{E:hcdlongrec}. This proves the three-term recurrences.

\end{proof}

\section{Proof for $H_1$}
\label{S:H1}
In this section we prove \eqref{C:H1}. We treat this case separately as the functional equation for $H_1(x)$ turns out to be more complicated than those for $H_j$ with $2\leq j\leq 9$.

We claim that
\begin{multline}\label{E:H1Rec}
H_1(x)= \left(1+q\left(1+q-q^2\right)x\right) H_1\left(xq^2\right) + xq^3\left(1-xq^2\left(1-q-q^2\right)\right)H_1\left(xq^4\right)\\
+ x^2 q^8\left(1-xq^4\right) H_1\left(xq^6\right).
\end{multline}
For this we write, with $N:=i+2j+3k$ throughout,
\begin{equation*}
(i+2j+3k)(i+2j+3k-1)+3k^2 + i +6j +6k = N^2 + 3k^2 +4j +3k,
\end{equation*}
to obtain
\begin{equation*}
H_1(x) = \sum_{N \geq 0}h_{1,N}\cdot q^{N^2}  x^N,
\end{equation*}
where $h_{1,N} := h_{\frac52,1,N}$.
A direct calculation shows that \eqref{E:H1Rec} is equivalent to
\begin{align}\label{rech1N}
\left(1 - q^{2N}\right)h_{1,N} = \left(1 + q - q^2 + q^{2N}\right)h_{1,N-1} + \left(-q + q^2 + q^3 + q^{2N}\right)h_{1,N-2} - q^3 h_{1,N-3},
\end{align}
which follows from \eqref{E:d+3/2rec} with $d=1$
(as it turns out, $H_1$ is the only identity that uses \eqref{E:d+3/2rec}).

The remainder of the proof is quite similar to that of Proposition \ref{Prop1}. Set
\begin{align*}
B(x) &:= \frac{H_1(x)}{(x;q^2)_\infty},
\end{align*}
so that dividing \eqref{E:H1Rec} by $(xq^4; q^2)_\infty$ yields
\begin{align*}
& (1-x) \left(1-xq^2\right) B(x)
\\
& =
\left(1 + xq + xq^2 - xq^3\right)\left(1-xq^2\right) B\left(xq^2\right)
+xq^3 \left(1 - xq^2 + xq^3 + xq^4\right)B\left(xq^4\right)
+ x^2 q^8 B\left(xq^6\right).
\end{align*}
Writing
\begin{equation*}
B(x) =:\sum_{n\ge0}\beta_nx^n
\end{equation*}
gives
\begin{align}\label{E:H1betarec}
\left(1 - q^{2n}\right) \beta_{n}
&=
\left(1 + q^2 + q^{2n-1} - q^{2n+1} + q^{4n-1}\right) \beta_{n-1}
-q^2 \left(1-q^{2n-1}\right)\left(1+q^{2n-2}\right)\left(1+q^{2n-3}\right)\beta_{n-2}.
\end{align}

Set
\begin{align*}
\gamma_{n} := \frac{\left(q^2; q^2\right)_n}{(q;q)_{2n+1}} \beta_{n},
\qquad\qquad\qquad
C(x) := \sum_{n\ge0}\gamma_n x^n.
\end{align*}
Proposition \ref{P:Appell}
applies to $B(x)$ and $C(x)$, giving that
\begin{align*}
\lim_{x\rightarrow 1^-} \left(1-x\right)B\left(x\right)
&=
\beta_{\infty}
=
\left(q;q^2\right)_\infty \gamma_\infty
=
\left(q;q^2\right)_\infty
\lim_{x\rightarrow 1^-} \left(1-x\right)C\left(x\right)
.
\end{align*}
Multiplying \eqref{E:H1betarec} by $\displaystyle \frac{(q^2;q^2)_{n-1}}{(q;q)_{2n-1}}$ yields
\begin{multline*}
\left(1-q^{2n}\right)\left(1-q^{2n+1}\right) \gamma_{n}
\\= \left(1 + q^2 + q^{2n-1} - q^{2n+1} + q^{4n-1}\right) \gamma_{n-1}
- q^2 \left(1+q^{2n-2}\right)\left(1+q^{2n-3}\right) \gamma_{n-2}.
\end{multline*}
In terms of $C(x)$, this is
\begin{align}\label{E:H1Crec}
(1-x)\left(1-xq^2\right)C(x)
&=
(1+q)(1+xq)\left(1-xq^2\right) C\left(xq^2\right) - q\left(1 - xq^2 + x^2 q^4\right) C\left(xq^4\right).
\end{align}

Set
\begin{equation*}
D(x) := \frac{\left(-x;q^2\right)_\infty \left(x; q^2\right)_\infty}{\left(-x^3; q^6\right)_\infty} C(x).
\end{equation*}
We multiply \eqref{E:H1Crec} by $\frac{(-xq^2;\,q^2)_\infty (xq^4; q^2)_\infty}{(-x^3 q^6; q^6)_\infty}$ to obtain
\begin{align*}
\frac{1+x^3}{1+x}D(x)
&=
(1+q)(1+xq)D\left(xq^2\right) - qD\left(xq^4\right).
\end{align*}
That is,
\begin{align*}
\left(1-x+x^2\right)D(x)
&=
\left(1+q+xq+xq^2\right)D\left(xq^2\right) - qD\left(xq^4\right).
\end{align*}
Writing
\begin{equation*}
D(x) =: \sum_{n\geq 0} \delta_nx^n
\end{equation*}
yields
\begin{align}\label{E:H1deltarec}
\left(1-q^{2n}\right)\left(1-q^{2n+1}\right)\delta_{n}
&=
\left(1+q^{2n-1}+q^{2n}\right)\delta_{n-1} - \delta_{n-2}
.
\end{align}
Note that $\delta_0=\gamma_0=\frac{1}{1-q}$.

Lastly, set
$$
\varepsilon_n := \left(q;q^2\right)_{n+1}\delta_{n}.
$$
Multiplying
\eqref{E:H1deltarec}
by $(q;q^2)_n$ gives
\begin{align*}
\left(1-q^{2n}\right)\varepsilon_{n}
&= \left(1+q^{2n-1}+q^{2n}\right)\varepsilon_{n-1}
- \left(1-q^{2n-1}\right)\varepsilon_{n-2}.
\end{align*}
Setting
\begin{equation*}
E(x) =: \sum_{n\ge0}\varepsilon_n x^n,
\end{equation*}
we obtain
\begin{align*}
E(x)&=\frac{(1+x)(1+xq)\left(1+xq^2\right)}{1+x^3}E\left(xq^2\right).
\end{align*}
Noting that $E(0)=\varepsilon_{0}=1$, we have that
\begin{align*}
E(x)
&=
\frac{(-x;q)_\infty\left(-xq^2;q^2\right)_\infty}{\left(-x^3;q^6\right)_\infty}.
\end{align*}

Using \eqref{E:euler}, we find that
\begin{align*}
E(x)
&=
\sum_{\ell,m,r\ge0}
\frac{(-1)^\ell q^{\frac12m(m-1) + r(r+1)} }
{(q;q)_m \left(q^2;q^2\right)_r \left(q^6;q^6\right)_\ell}x^{m+r+3\ell}.
\end{align*}
Thus
\begin{align*}
\varepsilon_{n}
&=
\sum_{r+m+3\ell=n}
\frac{(-1)^\ell q^{\frac12 m(m-1) + r(r+1)} }
{(q;q)_m \left(q^2;q^2\right)_r \left(q^6;q^6\right)_\ell}
,\qquad\quad
\delta_{n}
=
\sum_{r+m+3\ell=n}
\frac{(-1)^\ell q^{\frac12 m(m-1) + r(r+1)} }
{(q;q)_m \left(q^2;q^2\right)_r \left(q^6;q^6\right)_\ell \left(q;q^2\right)_{n+1}}
.
\end{align*}

We then have that
\begin{align}\notag
D(x)
&=
\sum_{\substack{n\ge0,\\ m+r+3\ell=n}}\!\!
\frac{(-1)^\ell  q^{\frac12 m(m-1)+r(r+1)} }
{(q;q)_m \left(q^2;q^2\right)_r \left(q^6;q^6\right)_\ell \left(q;q^2\right)_{n+1}}x^{n}
\\
&=
\sum_{n,r,\ell\ge 0}\!\!
\frac{(-1)^\ell q^{\frac12 n(n-1)+r(r+1)} }
{(q;q)_n \left(q^2;q^2\right)_r \left(q^6;q^6\right)_\ell \left(q;q^2\right)_{n+r+3\ell+1}} x^{n+r+3\ell}
,\label{E:H1eqD1}
\end{align}
changing $n\mapsto n+r+3\ell$ in the second line.
We apply \eqref{E:heine} to the inner sum on $r$, finding
it is equal to
\begin{align*}
&\frac{1}{\left(q;q^2\right)_{n+3\ell+1}}
\lim_{\substack{a\rightarrow \infty\\b\rightarrow 0 }}
{_2\phi_1}\Bigg[\begin{matrix}
a,b\\
q^{2n+6\ell+3}\end{matrix};q^2,\frac{-xq^2}{a}\Bigg]
\\
&=
	\frac{1}{\left(q;q^2\right)_{n+3\ell+1}}
	 \frac{\left(-xq^2;q^2\right)_\infty}{\left(q^{2n+6\ell+3};q^2\right)_\infty}
	\lim_{b\rightarrow 0}
	{_2\phi_1}\Bigg[\begin{matrix}
	\frac{q^{2n+6\ell+3}}{b}, 0\\
	-xq^2\end{matrix};q^2,b\Bigg]
=
	\frac{\left(-xq^2;q^2\right)_\infty}{\left(q;q^2\right)_\infty}
	\sum_{r\ge0}\frac{ (-1)^r q^{r^2+2r+2nr+6\ell r}  }
	{\left(-xq^2,q^2;q^2\right)_r}.
\end{align*}
Inserting this into \eqref{E:H1eqD1},
and evaluating the inner sums on $n$ and $\ell$ with \eqref{E:euler},
we have
\begin{align}\label{E:H1eqD2}
D(x)
&=
\frac{\left(-xq^2;q^2\right)_\infty}{\left(q;q^2\right)_\infty}
\sum_{\substack{n,r,\ell\ge 0}}
\frac{(-1)^{r+\ell}  q^{\frac12 n(n-1) +r^2+2r+2nr+6\ell r }  }
{(q;q)_n \left(-xq^2,q^2;q^2\right)_r \left(q^6;q^6\right)_\ell}x^{n+3\ell}
\nonumber\\
&=
\frac{\left(-xq^2;q^2\right)_\infty}{\left(q;q^2\right)_\infty}
\sum_{r\ge0}
	\frac{(-1)^r q^{r^2+2r}  }{ \left(-xq^2,q^2;q^2\right)_r }
\sum_{n\ge0}
	\frac{  q^{\frac12n(n-1) +2nr} }{(q;q)_n}x^n
\sum_{\ell\ge0}
	\frac{ (-1)^{\ell} q^{6\ell r} } {\left(q^6;q^6\right)_\ell}x^{3\ell}
\nonumber\\
&=
\frac{\left(-xq^2;q^2\right)_\infty}{\left(q;q^2\right)_\infty}
\sum_{r\ge0}
	\frac{(-1)^r q^{r^2+2r}  }{ \left(-xq^2,q^2;q^2\right)_r }
	\left(-xq^{2r};q\right)_\infty
	\frac{1} {\left(-x^3q^{6r};q^6\right)_\infty}
\nonumber\\
&=
\frac{(-x;q)_\infty\left(-xq^2;q^2\right)_\infty}{\left(-x^3;q^6\right)_\infty \left(q;q^2\right)_\infty }
\sum_{r\ge0}
\frac{(-1)^r \left(-x^3;q^6\right)_r q^{r^2+2r}  }{(-x;q)_{2r} \left(-xq^2,q^2;q^2\right)_r }
.
\end{align}

We claim that
\begin{align}\label{EqH1Final}
\sum_{r\ge0}\frac{(-1)^r \left(-1;q^6\right)_r q^{r^2+2r}  }{(-1;q)_{2r} \left(-q^2,q^2;q^2\right)_r }
&=
\frac{ \left(q;q^2\right)_\infty \left(q^2,q^{10},q^{12};q^{12}\right)_\infty  }{ \left(q^2;q^2\right)_\infty \left(q,q^{11};q^{12}\right)_\infty   }
.
\end{align}
This follows, letting $q\mapsto -q$ in the following identity of McLaughlin and Sills \cite[(1.12)]{MS}
\begin{align*}
\sum_{r\ge0} \frac{ \left(-1;q^6\right)_r (-q;q^2)_r q^{r^2+2r}  }{ (-1;q^2)_{r} (q^2;q^2)_{2r} }
&=
\frac{ \left(-q;q^2\right)_\infty \left(q,q^{11},q^{12};q^{12}\right)_\infty \left(q^{10},q^{14};q^{24}\right)_\infty  }
{ \left(q^2;q^2\right)_\infty}
.
\end{align*}
This yields that
\begin{equation*}
D(1) =
\frac{(-1;q)_\infty \left(-q^2;q^2\right)_\infty \left(q^2,q^{10},q^{12};q^{12}\right)_\infty}
{\left(-1;q^6\right)_\infty\left(q^2;q^2\right)_\infty\left(q,q^{11};q^{12}\right)_\infty}.
\end{equation*}

We now evaluate $H_1(1)$.
By Proposition \ref{P:Appell}, and recalling \eqref{E:H1eqD2} and \eqref{EqH1Final}, we have
\begin{align*}
H_1(1)
&=
\left(q^2; q^2\right)_\infty \lim_{x\rightarrow1^-}(1-x)B(x)
=
\left(q; q\right)_\infty
\lim_{x\rightarrow1^-}(1-x)C(x)
=
\frac{(q;q)_\infty \left(-1; q^6\right)_\infty}{\left(-1;q^2\right)_\infty \left(q^2; q^2\right)_\infty} D(1)
\\
&=
\frac{1}{\left(q,q^4, q^6, q^8, q^{11}; q^{12}\right)_\infty}.
\end{align*}
This completes the proof of \eqref{C:H1}.

\section{Proofs for $H_j$ with $2\leq j\leq 11$}
\label{S:H2-11}

In the following subsections, we give short explanations of the results for
these functions. The functions $H_j$ for $2\leq j\leq 9$ use Proposition \ref{Prop1},
while $H_{10}$ and $H_{11}$ require Proposition \ref{Prop2}.

\subsection{$H_2$}
As noted in \cite[Theorems 3.3.1 and 3.3.2]{KR18}, the identity for
$H_2(1)$ follows from the identities for $H_1(1)$ and $H_3(1)$. As such,
we skip $H_2(1)$.

\subsection{$H_3$}
We claim that
\begin{align}
\label{E:H3Rec}
H_3(x)=\left(1+xq^4\right) H_3\left(xq^2\right) + xq^5 \left(1+xq^3\right) H_3\left(xq^4\right) + x^2 q^{12} \left(1-xq^4\right) H_3\left(xq^6\right).
\end{align}
For this we write
\begin{align*}
(i+2j+3k)(i+2j+3k-1)+3k^2 + 4i+6j+12k = N^2 +3N +3k^2 -2j,
\end{align*}
to obtain that
\begin{align*}
H_3(x) = \sum_{N \geq 0} h_{3,N} \cdot  q^{N^2+3N} x^N,
\end{align*}
where $h_{3,N} := h_{-\frac12,0,N}$.
We find that \eqref{E:H3Rec} is equivalent to
\begin{align*}
\left(1-q^{2N}\right)h_{3,N} &= \left(1+q^{2N-1}\right) h_{3,N-1} + \left(q^{-2}+q^{2N-2}\right)h_{3,N-2} - q^{-2} h_{3,N-3}.
\end{align*}
This is \eqref{E:d=0rec} with $c = -\frac12.$

We next apply Proposition \ref{Prop1} with $a=2$ and $b=3$,
and find that
\begin{align*}
H_3(1)
&=
	\left(q^6;q^4\right)_\infty\sum_{n\ge0}
	\frac{ \left(q^3;q^6\right)_n q^{n^2+3n}}{\left(q,q^2;q^2\right)_{n} \left(q^6;q^4\right)_n}
.
\end{align*}
We rewrite the sum as
\begin{align*}
\sum_{n \geq 0} \frac{\left(-q^2; q^2\right)_n \left(q^3; q^6\right)_n q^{n^2 + 3n} }
{\left(q; q^2\right)_n \left(q^4; q^4\right)_n \left(q^6; q^4\right)_n}
&=
\left(1-q^2\right) \sum_{n \geq 0} \frac{ \left(-q^2; q^2\right)_n \left(q^3; q^6\right)_n q^{n^2 + 3n}}
{\left(q;q^2\right)_n \left(q^2; q^2\right)_{2n+1}}\\
&= \frac{1}{\left(q^4, q^5, q^6, q^6, q^7, q^8, q^{10},q ^{14}; q^{12}\right)_\infty},
\end{align*}
using identity (1.30) of \cite{MS}. This proves \eqref{C:H3}.

\subsection{$H_4$}
We claim that
\begin{align}
\label{E:H4Rec}
H_4(x)=(1+xq) H_4\left(xq^2\right) + xq^2 \left(1+xq^3\right) H_4\left(xq^4\right) + x^2 q^6 \left(1-xq^4\right) H_4\left(xq^6\right).
\end{align}
For this we write
\begin{align*}
(i+2j+3k)(i+2j+3k-1)+3k^2+i+3j+3k = N^2+3k^2+j
\end{align*}
to obtain that
\begin{align*}
H_4(x) = \sum_{N \geq 0}  h_{4,N} \cdot  q^{N^2} x^N,
\end{align*}
where $h_{4,N} := h_{1,0,N}$.
We find that \eqref{E:H4Rec} is equivalent to
\begin{align*}
\left(1-q^{2N}\right) h_{4,N} &= \left(1+q^{2N-1}\right)h_{4,N-1} +\left(q+q^{2N-2}\right)h_{4,N-2}-q h_{4,N-3}.
\end{align*}
This is \eqref{E:d=0rec} with $c = 1$.  Proposition \ref{Prop1} with $a=-1$ and $b=3$
then implies \eqref{T:H4}.

\subsection{$H_5$}
We claim that
\begin{align}
\label{E:H5Rec}
H_5(x)=\left(1+xq^3\right) H_5\left(xq^2\right) + xq (q+x) H_5\left(xq^4\right) + x^2 q^6 \left(1-xq^2\right) H_5\left(xq^6\right).
\end{align}
For this we write
\begin{align*}
	(i+2j+3k)(i+2j+3k-1)+3k^2+2i-j+3k = N^2+N+3k^2-5j-3k
\end{align*}
to obtain that
\begin{align*}
H_5(x) = \sum_{N \geq 0 }  h_{5,N} \cdot q^{N^2+N} x^N,
\end{align*}
where $h_{5,N} := h_{-2,-1,N}$.
We find that \eqref{E:H5Rec} is equivalent to
\begin{align*}
\left(1-q^{2N}\right) h_{5,N} &= \left(q+q^{2N-2}\right) h_{5,N-1} + \left(q^{-5}+q^{2N-4}\right)h_{5,N-2}-q^{-4}h_{5,N-3}.
\end{align*}
This is \eqref{E:d=-1rec} with $c = -2$.

Proposition \ref{Prop1} does not directly apply to $H_5(x)$,
and unlike with $H_1(x)$, we cannot adapt the proof.
In particular, the first reasonable step would be to
divide the functional equation for $H_5(x)$ by $(xq^{-2};q^2)_\infty$,
but then it is not valid to apply Appell's Comparison Theorem,
because the resulting series has radius of convergence $|q|^2<1$.
As such, we introduce another function that is related to $H_5(x)$
and to which Proposition \ref{Prop1} does apply.
We set
\begin{align*}
J_5\left(x\right) &:= H_5\left(x\right)-xq^3H_5\left(xq^2\right).
\end{align*}
The idea is to apply Proposition \ref{Prop1} to $H_5(xq^2)$
and $J_5(x)$ separately and obtain from this a formula for $H_5(1)$.

From \eqref{E:H5Rec}, we obtain that
\begin{align*}
H_5\left(xq^2\right)
&=
	\left(1+xq^5\right)H_5\left(xq^4\right)
	+xq^4(1+xq)H_5\left(xq^6\right)
	+x^2q^{10}\left(1-xq^4\right)H_5\left(xq^8\right)
,
\end{align*}
and so applying Proposition \ref{Prop1} with $a=3$ and $b=1$ to $A(x)= H_5(xq^2)$ gives that
\begin{align*}
H_5\left(q^2\right)
&=
	\left(q^3;q^4\right)_\infty
	\sum_{n\ge0}
	\frac{\left(q^{-3};q^6\right)_n  q^{ n^2 + 4n}}
	{\left(q^{-1},q^2;q^2\right)_{n}  \left(q^{3};q^4\right)_n }
.
\end{align*}

We claim that
\begin{align*}
J_5\left(x\right)
&=
	\left(1+xq^5\right)J_5\left(xq^2\right)
	+xq\left(q+x\right)J_5\left(xq^4\right)
	+x^2q^6\left(1-xq^4\right)J_5\left(xq^6\right)
.
\end{align*}
This follows from the functional equation for $H_5$
by shifting $x \mapsto xq^2$ in \eqref{E:H5Rec}, multiplying by $xq^5$, and then subtracting
the resulting equation from \eqref{E:H5Rec}.

We apply Proposition \ref{Prop1} with $a=3$ and $b=-1$ to find that
\begin{align*}
J_5(1)
&=
	\left(q^{-1};q^4\right)_\infty
	\sum_{n\ge0}
	\frac{\left(q^{-9};q^6\right)_n  q^{ n^2 + 4n}  }
	{(q^{-3},q^2;q^2)_{n}  \left(q^{-1};q^4\right)_n }
.
\end{align*}
Thus
\begin{align*}
H_5(1)
&=
	\left(q^{-1};q^4\right)_\infty
	\sum_{n\ge0}
	\frac{ \left(q^{-9};q^6\right)_n  q^{ n^2 + 4n} }
	{(q^{-3},q^2;q^2)_{n}  \left(q^{-1};q^4\right)_n }
	+
	\left(q^3;q^4\right)_\infty
	\sum_{n\ge0}
	\frac{  \left(q^{-3};q^6\right)_n q^{ n^2 + 4n+3} }
	{(q^{-1},q^2;q^2)_{n}  \left(q^{3};q^4\right)_n }
.
\end{align*}
By isolating the $n=0$ term and then shifting $n\mapsto n+1$, we find that
\begin{align*}	
\sum_{n\ge0}
\frac{  \left(q^{-9};q^6\right)_n  q^{ n^2 + 4n}}
{(q^{-3},q^2;q^2)_{n}  \left(q^{-1};q^4\right)_n }
&=
1+\frac{1+q^{-3}+q^{-6}}{1-q^{-1}}
\sum_{n\ge0}
\frac{ \left(q^{-3};q^6\right)_n q^{ n^2 + 6n+5}  }
{\left(1-q^{2n+2}\right)\left(q^{-1},q^2;q^2\right)_{n}  \left(q^3;q^4\right)_n }
,
\end{align*}
Thus
\begin{align*}
H_5(1)
&=
	\left(q^{-1};q^4\right)_\infty
	+
	\left(q^3;q^4\right)_\infty
	\sum_{n\ge0}
	\frac{\left(1+q^{2n-4}+q^{2n-1}\right) \left(q^{-3};q^6\right)_n q^{ n^2 + 4n+3} }
	{\left(1-q^{2n+2}\right)\left(q^{-1},q^2;q^2\right)_{n}  \left(q^{3};q^4\right)_n }
.
\end{align*}
This proves \eqref{T:H5}.

\subsection{$H_6$}
We claim that
\begin{align}
\label{E:H6Rec}
H_6(x)=\left(1+xq^2\right) H_6\left(xq^2\right) + xq \left(1+xq\right) H_6\left(xq^4\right) + x^2 q^4 \left(1-xq^4\right) H_6\left(xq^6\right).
\end{align}
For this we write
\begin{align*}
	(i+2j+3k)(i+2j+3k-1)+3k^2+i = N^2+3k^2-3k-2j
\end{align*}
to obtain that
\begin{align*}
H_6(x) = \sum_{N \geq 0}  h_{6,N} \cdot q^{N^2} x^N,
\end{align*}
where $h_{6,N} := h_{-\frac12,-1,N}$.
We find that \eqref{E:H6Rec} is equivalent to
\begin{align*}
\left(1-q^{2N}\right)h_{6,N}&=\left(q+q^{2N-2}\right)h_{6,N-1} +\left(q^{-2}+q^{2N-4}\right)h_{6,N-2} - q^{-1}h_{6,N-3}.
\end{align*}
This is \eqref{E:d=-1rec} with $c = -\frac12$. Proposition \ref{Prop1} with $a=0$ and $b=1$ then implies \eqref{C:H6}.

\subsection{$H_7$}
We claim that
\begin{align}
\label{E:H7Rec}
H_7(x)=\left(1+xq^2\right) H_7\left(xq^2\right) + xq^3 \left(1+xq^3\right) H_7\left(xq^4\right) + x^2 q^8 \left(1-xq^4\right) H_7\left(xq^6\right).
\end{align}
For this we write
\begin{align*}
	(i+2j+3k)(i+2j+3k-1)+3k^2+2i+4j+6k = N^2+N+3k^2
\end{align*}
to obtain that
\begin{align*}
H_7(x) = \sum_{N \geq 0} h_{7,N} \cdot q^{N^2+N} x^N,
\end{align*}
where
$h_{7,N} := h_{\frac12,0,N}$.
We find that \eqref{E:H7Rec} is equivalent to
\begin{align*}
\left(1-q^{2N}\right)h_{7,N} &= \left(1+q^{2N-1}\right)h_{7,N} + \left(1+q^{2N-2}\right)h_{7,N-2}-h_{7,N-3}.
\end{align*}
This is \eqref{E:d=0rec} with $c = \frac12$.
Proposition \ref{Prop1} with $a=0$ and $b=3$ then implies \eqref{C:H7}.

\subsection{$H_9$}

We next consider $H_9$, as we subsequently need these calculations for $H_8$.
We claim that
\begin{equation}
\label{E:H9Rec}
H_9(x) = \left(1+xq^3\right) H_9\left(xq^2\right) + xq^{4}\left(1+xq^3\right) H_9\left(xq^4\right) + x^2 q^{10} \left(1-xq^4\right) H_9\left(xq^6\right).
\end{equation}
For this, we write
\begin{align*}
	(i+2j+3k)(i+2j+3k-1)+3k^2+3i+5j+9k = N^2+2N+3k^2-j
\end{align*}
to obtain that
\begin{align*}
H_9(x)= \sum_{N \geq 0}  h_{9,N} \cdot q^{N^2+2N} x^N,
\end{align*}
where $h_{9,N} := h_{0,0,N}$.
We find that \eqref{E:H9Rec} is equivalent to
\begin{align*}
\left(1-q^{2N}\right)h_{9,N}
=\left(1+ q^{2N-1}\right) h_{9,N-1} + \left(q^{-1}+ q^{2N-2} \right)h_{9,N-2}  - q^{-1} h_{9,N-3}.
\end{align*}
This is \eqref{E:d=0rec} with $c=0$.
Proposition \ref{Prop1} with $a=1$ and $b=3$ then implies
\eqref{T:H9}.

\subsection{$H_8$}

We claim that
\begin{align}
\label{E:H8rec}
H_8(x)= (1+xq) H_8\left(xq^2\right) + xq^2(1+xq) H_8\left(xq^4\right) + x^2q^6\left(1-xq^2\right) H_8\left(xq^6\right).
\end{align}
For this we write
\begin{align*}
(i+2j+3k)(i+2j+3k-1)+3k^2+i+j+3k = N^2+3k^2-j
\end{align*}
to obtain that
\begin{align*}
H_8(x)= \sum_{ N \geq 0}   h_{8,N} \cdot q^{N^2} x^N,
\end{align*}
where $h_{8,N} := h_{9,N}$
We find that \eqref{E:H8rec} is equivalent to
\begin{align*}
\left(1-q^{2N}\right)h_{8,N} = \left(1+ q^{2N-1}\right) h_{8,N-1} + \left(q^{-1}+ q^{2N-2}\right) h_{8,N-2} - q^{-1} h_{8,N-3}.
\end{align*}
Note that $h_{8,N} = h_{9,N}$ and so the functional equation is the same.

As with $H_5$, we cannot apply Proposition \ref{Prop1} directly.
We set
\begin{align*}
J_8\left(x\right) &:= H_8\left(x\right)-xqH_8\left(xq^2\right),
\end{align*}
and note that $H_8(xq^2)=H_9(x)$. Equation \eqref{T:H9} then gives a
$q$-hypergeometric series representation for $H_8(q^2)$. To find a $q$-hypergeometric series
representation for $J_8(1)$, we use the recurrence
\begin{align*}
J_8\left(x\right)
&=
	\left(1+xq^3\right)J_8\left(xq^2\right)
	+xq^2\left(1+xq\right)J_8\left(xq^4\right)
	+x^2q^6\left(1-xq^4\right)J_8\left(xq^6\right)
.
\end{align*}
This follows by setting $x \mapsto xq^2$ in \eqref{E:H8rec}, multiplying by $xq^3$, and
then subtracting the resulting equation from \eqref{E:H8rec}.

We apply Proposition \ref{Prop1} with $a=b=1$ to find that
\begin{align*}
J_8(1)
&=
	\left(q;q^4\right)_\infty
	\sum_{n\ge0}
	\frac{ \left(q^{-3};q^6\right)_n  q^{ n^2 + 2n} }
	{\left(q^{-1},q^2;q^2\right)_{n}  \left(q;q^4\right)_n }
.
\end{align*}
Thus
\begin{align*}
H_8(1)
&=
	J_8(1)+qH_9(1)
=
	\left(q;q^4\right)_\infty
	\sum_{n\ge0}
	\frac{ \left(q^{-3};q^6\right)_n   q^{ n^2 + 2n}}
	{\left(q^{-1},q^2;q^2\right)_{n}  \left(q;q^4\right)_n }
	+
	\left(q^5;q^4\right)_\infty
	\sum_{n\ge0}
	\frac{ \left(q^{3};q^6\right)_n  q^{ n^2 + 2n+1} }
	{\left(q,q^2;q^2\right)_{n}  \left(q^5;q^4\right)_n }
.
\end{align*}
To finish the prove we rewrite the first sum, splitting off the $n=0$ term, as
\begin{equation*}
\left(q;q^4\right)_\infty + \left(q^5;q^4\right)_\infty \frac{1-q^{-3}}{1-q^{-1}} \sum_{n \geq 0} \frac{ \left(q^3;q^6\right)_{n}q^{n^2 + 4n + 3}}{\left(q;q^2\right)_{n} (q^2; q^2)_{n+1} \left(q^5; q^4\right)_{n}}.
\end{equation*}
Combining proves \eqref{T:H8}.

\subsection{$H_{10}$}

To begin we rewrite, using \eqref{E:euler},
\begin{align*}
&H_{10}(1)
=
\sum_{i,j,k\ge0}
\frac{ q^{\frac12i(i-1)+i+2ij+3ik +\frac12(2j+3k)(2j+3k-1)+j^2+2j+4k} }
{(q;q)_i(q^2;q^2)_j(q^3;q^3)_k}
\\
&=
\sum_{j,k\ge0}
\frac{ \left(-q^{1+2j+3k};q\right)_\infty q^{\frac12(2j+3k)(2j+3k-1)+j^2+2j+4k} }
{(q^2;q^2)_j(q^3;q^3)_k}
=
(-q;q)_\infty
\sum_{j,k\ge0}
\frac{ q^{\frac12(2j+3k)(2j+3k-1)+j^2+2j+4k} }
{(-q;q)_{2j+3k}(q^2;q^2)_j(q^3;q^3)_k}
.
\end{align*}
We set
\begin{align*}
J_{10}(x)
&:=
\sum_{j,k\ge0}
\frac{q^{\frac12(2j+3k)(2j+3k-1)+j^2+2j+4k} }
{(-q;q)_{2j+3k}(q^2;q^2)_j(q^3;q^3)_k}x^{2j+2k}
,
\end{align*}
and claim that
\begin{align}\label{E:J10Rec}
J_{10}(x)
&=
	 q^{-2}\left(1+q^2\right)\left(1+x^2q^6\right)J_{10}\left(xq^3\right)
	 -q^{-2}\left(1+x^2q^7\right)\left(1+x^2q^{11}\right)J_{10}\left(xq^6\right)
.	
\end{align}
To prove \eqref{E:J10Rec} we write, with $M:=j+k$ throughout,
\begin{equation*}
\frac12(2j+3k)(2j+3k-1)+j^2+ 2j + 4k = 3M^2+M+\frac{3k^2}{2}+\frac{3k}{2}.
\end{equation*}
This yields that
\begin{equation*}
J_{10}(x) = \sum_{M\geq 0}  j_{10,M} \cdot q^{3M^2+M} x^{2M},
\end{equation*}
where
\begin{equation*}
j_{10,M} := \sum_{k\geq 0} \frac{q^{\frac32 k(k+1)}}{(-q;q)_{2M+k} \left(q^2;q^2\right)_{M-k} \left(q^3;q^3\right)_{k}},
\end{equation*}
which also gives that $j_{10,M} = 0$ for $M < 0$, and $j_{10,0} = 1$.
We find that \eqref{E:J10Rec} is equivalent to
\begin{equation}
\label{E:j10rec}
\left(1 - q^{6M}\right) \left(1 - q^{6M-2}\right) j_{10,M} =
\left(1 + q^2 - q^{6M-5} - q^{6M-1}\right) j_{10,M-1} - q^2 j_{10, M-2}.
\end{equation}
\sloppy
We prove \eqref{E:j10rec} using the $q$-Zeilberger algorithm, as
implemented in MAPLE's \texttt{QDifferenceEquations} package. Set
\begin{align*}
f_{10,k,M} &:=
\frac{q^{\frac32 k(k+1)}}{\left(-q;q\right)_{2M+k}\left(q^2;q^2\right)_{M-k}\left(q^3;q^3\right)_{k}}
,\\
g_{10,k,M}
&:=
	\frac{q^6\left( q^{4M+2} - q^{4M+3k-1}(1+q) - q^{6M+4k+2} + q^{6M+k+3}(1+q) + q^{8M+2k+6}(1-q^2)\right)}
		 {\left(1+q^{2M+k+3}\right)\left(1+q^{2M+k+2}\right)\left(1+q^{2M+k+1}\right)\left(q^{2k}-q^{2M+2}\right)\left(q^{2M+4}-q^{2k}\right)}
	\\&\quad\times
	\left(1-q^{3k}\right)f_{10,k,M}
.
\end{align*}
Elementary rearrangements then reveal that
\begin{align}\label{E:H10WZpair}
&\left(1-q^{6M+10}\right)\left(1-q^{6M+12}\right)f_{10,k,M+2}
-\left(1+q^2-q^{6M+7}-q^{6M+11}\right)f_{10,k,M+1}
+q^2f_{10,k,M}
\nonumber\\
&= g_{10,k+1,M} - g_{10,k,M}.
\end{align}
We note that $g_{10,0,M}=0$ and $\lim_{k\rightarrow\infty}g_{10,k,M}=0$,
so that summing \eqref{E:H10WZpair} over $k$ implies
\eqref{E:j10rec}. The $q$-Zeilberger algorithm is an effective tool for verifying the recurrence satisfied by $J_{10}$
(and $J_{11}$, as seen in the next subsection) as $j_{10,M}$ has just one summation variable and one additional parameter $M$.
\fussy

Proposition \ref{Prop2} applied to $J_{10}$ with
$a=-2$, $b=4$, $c=6$, $\alpha_0=1$, and $\alpha_1=0$
implies that
\begin{align*}
H_{10}(1)
&=
(-q;q)_\infty J_{10}(1)
=
(-q;q)_\infty\left(-q^5;q^6\right)_\infty
{_2\phi_1}\left[\begin{matrix} q^{-1}, q\\ q^4 \end{matrix};q^6,-q^5\right].
\end{align*}
The $_2\phi_1$ can then be evaluated using \eqref{E:kummer} (with $a = q^{-1}, b=q,$ and $q \mapsto q^6$), and we obtain
\begin{align*}
H_{10}(1)
&=
\frac{\left(-q;q\right)_\infty \left(-q^5;q^6\right)_\infty \left(q^5,q^9,q^{12};q^{12}\right)_\infty }{\left(-q^5,q^4,q^{6};q^{6}\right)_\infty } =\frac{1}{\left(q;q^3\right)_\infty \left(q^3,q^{6},q^{11};q^{12}\right)_\infty }
.
\end{align*}
This proves \eqref{C:H10}.

\subsection{$H_{11}$}
As with $H_{10}$, by \eqref{E:euler},
we have that
\begin{align*}
H_{11}(1)
&=
\left(-q^2;q\right)_\infty
\sum_{j,k\ge0}
\frac{ q^{\frac12(2j+3k)(2j+3k-1)+j^2+4j+5k} }
{(-q^2;q)_{2j+3k}(q^2;q^2)_j(q^3;q^3)_k}
.
\end{align*}
We set
\begin{align*}
J_{11}(x)
:=
\sum_{j,k\ge0}
\frac{ q^{\frac12(2j+3k)(2j+3k-1)+j^2+4j+5k} }
{(-q^2;q)_{2j+3k}(q^2;q^2)_j(q^3;q^3)_k}x^{2j+2k}
\end{align*}
and claim that
\begin{align}\label{E:J11Rec}
J_{11}(x)
&=
	\left(1+q^2\right)\left(1+x^2q^6\right)J_{11}\left(xq^3\right)
	 -q^{2}\left(1+x^2q^7\right)\left(1+x^2q^{11}\right)J_{11}\left(xq^6\right)
.	
\end{align}
To show \eqref{E:J11Rec}, we write
\begin{equation*}
\frac12(2j+3k)(2j+3k-1)+j^2+ 4j + 5k = 3M^2+3M+\frac{3k^2}{2}+\frac{k}{2}.
\end{equation*}
Thus
\begin{equation*}
J_{11}(x) = \sum_{M\geq 0}  j_{11,M} \cdot q^{3M^2+3M} x^{2M},
\end{equation*}
where
\begin{equation*}
j_{11,M} := \sum_{k\geq 0} \frac{q^{\frac12 k(3k+1)}}{(-q^2;q)_{2M+k} \left(q^2;q^2\right)_{M-k} \left(q^3;q^3\right)_{k}}.
\end{equation*}
We find that \eqref{E:J11Rec} is equivalent to
\begin{equation}\label{E:j11rec}
\left(1 - q^{6M}\right) \left(1 - q^{6M+2}\right) j_{11,M} =
\left(1 + q^2 - q^{6M-3} - q^{6M+1}\right) j_{11,M-1} - q^2 j_{11, M-2}.
\end{equation}
As with $j_{10,M}$, we prove this recurrence using the $q$-Zeilberger algorithm.
Set
\begin{align*}
f_{11,k,M} &:=
\frac{q^{\frac12 k(3k+1)}}{\left(-q;q\right)_{2M+k+1}\left(q^2;q^2\right)_{M-k}\left(q^3;q^3\right)_{k}}
,\\
g_{11,k,M}
&:=
	\frac{q^7\left( q^{4M+1} - q^{4M+3k}\left(1+q^2\right) - q^{6M+4k+4} + q^{6M+k+4}(1+q) + q^{8M+2k+8}\left(1-q^2\right)\right)}
		 {\left(1+q^{2M+k+4}\right)\left(1+q^{2M+k+3}\right)\left(1+q^{2M+k+2}\right)\left(q^{2k}-q^{2M+2}\right)\left(q^{2M+4}-q^{2k}\right)}
	\\&\quad\times
	\left(1-q^{3k}\right)f_{11,k,M}
.
\end{align*}
Elementary rearrangements reveal that
\begin{align}\label{E:H11WZpair}
&\left(1-q^{6M+12}\right)\left(1-q^{6M+14}\right)f_{11,k,M+2}
-\left(1+q^2-q^{6M+9}-q^{6M+13}\right)f_{11,k,M+1}
+q^2f_{11,k,M}
\nonumber\\
&= g_{11,k+1,M} - g_{11,k,M}.
\end{align}
We note that $g_{11,0,M}=0$, $\lim_{k\rightarrow\infty}g_{11,k,M}=0$,
and $j_{11,M}=(1+q)\sum_{k\ge0}f_{11,k,M}$,
so that summing \eqref{E:H11WZpair} over $k$ implies
\eqref{E:j11rec}.

Proposition \ref{Prop2} with
$a=2$, $b=6$, $c=8$, $\alpha_0=1$, and $\alpha_1=0$
implies that
\begin{align*}
H_{11}(1)
&=
(-q^2;q)_\infty J_{11}(1)
=
\left(-q^2;q\right)_\infty \left(-q^5;q^6\right)_\infty
{_2\phi_1}\left[\begin{matrix} q^3, q \\ q^8 \end{matrix};q^6,-q^5\right].
\end{align*}
The $_2\phi_1$ again evaluates to a product by \eqref{E:kummer}, and in particular,
\begin{align*}
H_{11}(1)
=
\frac{\left(-q^2;q\right)_\infty \left(-q^5;q^6\right)_\infty \left(q^9,q^{13},q^{12};q^{12}\right)_\infty }{\left(-q^5,q^8,q^{6};q^{6}\right)_\infty }
&=
\frac{1}{\left(q^2;q^3\right)_\infty \left(q^3,q^{6},q^{7};q^{12}\right)_\infty }
.
\end{align*}
This proves \eqref{C:H11}.

\section{A Partial Reduction of Another Conjecture}
\label{S:J12}

In \cite[Section 5.2.1]{KR18}, Kanade and Russell also conjectured that
\begin{align*}
\sum_{i,j,k\geq 0}
\frac{ \left(1+q^{2i+4j+6k+2}-q^{3i+6j+9k+5}\right)q^{\frac12(i+2j+3k)(i+2j+3k-1)+i+j^2+2j+2k} }
	{(q;q)_i\left(q^2;q^2\right)_j\left(q^3;q^3\right)_k}
&= \frac{1}{\left(q^2;q^3\right)_\infty \left(q,q^6,q^9;q^{12}\right)_\infty}
.
\end{align*}
This combinatorial form of this conjecture was originally stated as Identity $I_{5a}$ in \cite{Rus}.
In this section we separate the sum into three series, and reduce two of them to single series. We set
\begin{align*}
J_{12,a}(x)
&:=
	\sum_{j,k\ge0} \frac{ q^{\frac12(2j+3k)(2j+3k-1)+j^2+(2a+2)j+(3a+2)k }}
	{(-q^{a+1};q)_{2j+3k}(q^2;q^2)_j(q^3;q^3)_k}x^{2j+2k}
,
\end{align*}
and rewrite the conjecture, using \eqref{E:euler} for each summand of the left-hand side of the conjecture, as
\begin{align*}
(-q;q)_\infty J_{12,0}(1)+q^2(-q^3;q)_\infty J_{12,2}(1) - q^5(-q^4;q)_\infty J_{12,3}(1)
&=
\frac{1}{\left(q^2;q^3\right)_\infty \left(q,q^6,q^9;q^{12}\right)_\infty}
.
\end{align*}

We claim that
\begin{align}
\label{E:H12J0Rec}
J_{12,0}(x)
&=
	\left(1+q^{-2}+x^2q^4+x^2q^8\right)J_{12,0}\left(xq^3\right)
	 -q^{-2}\left(1+x^2q^9\right)\left(1+x^2q^{11}\right)J_{12,0}\left(xq^{6}\right)
,\\
\label{E:H12J2Rec}
J_{12,2}(x)
&=
	\left(1+q^{4}+x^2q^8+x^2q^{10}\right)J_{12,2}\left(xq^3\right)
	 -q^4\left(1+x^2q^9\right)\left(1+x^2q^{11}\right)J_{12,2}\left(xq^{6}\right)
.
\end{align}
Writing
\begin{align*}
\frac12(2j+3k)(2j+3k-1)+ j^2 +(2a+2)j+(3a+2)k
&=
	3M^2+(2a+1)M+\frac12 k(3k+2a-1)
,
\end{align*}
we obtain that
\begin{align*}
J_{12,a}(x)
&=
	\sum_{M\ge0}j_{12,a,M} \cdot q^{3M^2+(2a+1)M} x^{2M}
,
\end{align*}
where
\begin{align*}
j_{12,a,M} &:= \sum_{k\ge0}\frac{q^{\frac12k(3k+2a-1)}}
	{\left(-q^{a+1};q\right)_{2M+k} \left(q^{2};q^2\right)_{M-k} \left(q^3;q^3\right)_{k} },
\end{align*}
which gives that $j_{12,a,M} = 0$ for $M<0$ and $j_{12,a,0} = 1$.
We find that \eqref{E:H12J0Rec} and \eqref{E:H12J2Rec} are
equivalent to, respectively,
\begin{align}
\label{E:j120rec}
\left(1-q^{6M-2}\right)\left(1-q^{6M}\right)j_{12,0,M}	
&=
\left(1+q^4-q^{6M-3}-q^{6M-1}\right)j_{12,0,M-1}	
-q^4j_{12,0,M-1}	
,\\
\label{E:j122rec}
\left(1-q^{6M}\right)\left(1-q^{6M+4}\right)j_{12,2,M}	
&=
\left(1+q^2-q^{6M-1}-q^{6M+1}\right)j_{12,2,M-1}	
-q^2j_{12,2,M-1}	
.
\end{align}

Again, these recurrences may be proved with the $q$-Zeilberger algorithm.
Set
\begin{align*}
f_{12,a,k,M}
&:=
	\frac{q^{\frac12 k(3k+2a-1)}}
	{\left(-q;q\right)_{2M+k+a} \left(q^{2};q^2\right)_{M-k} \left(q^3;q^3\right)_{k} }
,\\
g_{12,0,k,M}
&:=
	\frac{q^{10}\left(q^{2M+2k-6}(1-q^2) + q^{4M} - q^{4M+3k-2}(1+q) + q^{6M+k+2}(1+q) - q^{6M+4k}\right)}
		 {\left(1+q^{2M+k+3}\right)\left(1+q^{2M+k+2}\right)\left(1+q^{2M+k+1}\right)\left(q^{2k}-q^{2M+2}\right)\left(q^{2M+4}-q^{2k}\right)}
	\\&\qquad\times
	\left(1-q^{3k}\right)f_{12,0,k,M}
,\\
g_{12,2,k,M}
&:=
	\frac{q^{6}\left(q^{4M+2} -  q^{4M+3k+3}(1+q) + q^{6M+k+6}(1+q) - q^{6M+4k+8}\right) \left(1-q^{3k}\right)}
		 {\left(1+q^{2M+k+5}\right)\left(1+q^{2M+k+4}\right)\left(1+q^{2M+k+3}\right)\left(q^{2M+2}-q^{2k}\right)\left(q^{2k}-q^{2M+4}\right)}
	f_{12,2,k,M}
,
\end{align*}
so that
\begin{align}
\label{E:J120WZ}
&\left(1-q^{6M+10}\right)\left(1-q^{6M+12}\right)f_{12,0,k,M+2}
-\left(1+q^4-q^{6M+9}-q^{6M+11}\right)f_{12,0,k,M+1}	
+q^4f_{12,0,k,M}	
\nonumber\\
&=
g_{12,0,k+1,M}-g_{12,0,k,M}
,\\
\label{E:J122WZ}
&\left(1-q^{6M+12}\right)\left(1-q^{6M+16}\right)f_{12,2,k,M+2}	
-\left(1+q^2-q^{6M+11}-q^{6M+13}\right)f_{12,2,k,M+1}	
+q^2f_{12,2,k,M}	
\nonumber\\
&=
g_{12,2,k+1,M}-g_{12,2,k,M}
.
\end{align}
We note that $g_{12,0,0,M}=g_{12,2,0,M}=0$ and $\lim_{k\rightarrow\infty}g_{12,0,k,M}=\lim_{k\rightarrow\infty}g_{12,2,k,M}=0$,
so that summing \eqref{E:J120WZ} over $k$ implies \eqref{E:j120rec} and
summing \eqref{E:J122WZ} over $k$  implies \eqref{E:j122rec}.

We apply Proposition \ref{Prop2} with
$a=-2, \ b=4, \ c=8, \ \alpha_0=1$, and $\alpha_1=0$ to $J_{12,0}(x)$ and
with $a=4, \ b=8, \ c=10, \ \alpha_0=1$, and $\alpha_1=0$ to $J_{12,2}(x)$, to obtain,
\begin{align*}
J_{12,0}(1)
&=
	\left(-q^5;q^6\right)_\infty {_2\phi_1}\left[\begin{matrix} q^{-1},q^3 \\ q^4\end{matrix};q^6,-q^5\right]
,&
J_{12,2}(1)
&=
	\left(-q^5;q^6\right)_\infty {_2\phi_1}\left[\begin{matrix} q^3,q^5 \\ q^{10}\end{matrix};q^6,-q^5\right]
.
\end{align*}
However, we do not have a reduction for $J_{12,3}(1)$ as a basic hypergeometric series.

\section{Concluding Remarks}
\label{S:Conc}

Noting that Proposition \ref{Prop1} easily gives an infinite product whenever $a \!=\!0$, one might ask whether this leads to more identities of the form \eqref{C:H1} -- \eqref{C:H9}. In particular, consider the $q$-difference equation from Proposition \ref{Prop1} with $a = 0$, namely
\begin{equation}
\label{E:a=0qDif}
H(x) = \left(1+xq^2\right) H\left(xq^2\right) + xq^b \left(1+xq^b\right) H\left(xq^4\right) + x^2 q^{2b+2} \left(1-xq^4\right) H\left(xq^6\right).
\end{equation}
Proposition \ref{Prop1} implies that
\begin{align*}
H(1) = \frac{\left(q^{3b}; q^{12}\right)_\infty}{\left(q^2, q^b; q^4\right)_\infty}.
\end{align*}
In order to determine when this corresponds to a triple sum of the shape found in $H_j$ with $1\leq j \leq 9$, we compare to Proposition \ref{Prop2}.

Suppose that
\begin{align*}
H(x) = \sum_{N\geq 0}  h_N \cdot q^{N^2 + mN} x^N.
\end{align*}
Then \eqref{E:a=0qDif} is equivalent to
\begin{align*}
\left(1 - q^{2N}\right) h_N & = \left(q^{1-m} + q^{2N+b-3-m}\right) h_{N-1}
+ \left(q^{2b-4-2m} + q^{2N+2b-6-2m}\right)h_{N-2} - q^{2b-3-3m}h_{N-3}.
\end{align*}
This matches \eqref{E:d=0rec} only when $m=1$ and $b=3$, which is $H_6$, and matches \eqref{E:d=-1rec} when $m=0$ and $b=1$, which is $H_7$. Furthermore, a short calculation shows that there are no other cases of this shape. In particular, in the proof of Proposition \ref{P:hNcdrec},
one can solve for $a, A_1(N),$ and $A_2(N)$ by comparing \eqref{E:hgeniter} and \eqref{E:hcdlongrec}, but this turns out to only be possible for the three cases stated in the proposition.

We have also found another conjectural sum-product identity for a series of the form \eqref{E:H1-9def}. After searching for related identities,
we observed computationally that
\begin{align*}
\sum_{i,j,k\ge0}\frac{(-1)^{k} q^{(i+2j+3k)(i+2j+3k-1)+3k^2+i-3j-3k}}{(q;q)_i(q^4;q^4)_j(q^6;q^6)_k}
&=
\frac{q^{-1}\left(1+q+q^2\right)}{(q^3;q^4)_\infty (q^4,q^5;q^{12})_\infty}.
\end{align*}
In fact, it is not difficult to make the connection with $H_9$ explicit.
In particular, we find that
\begin{align*}
\sum_{i,j,k\ge0}\frac{(-1)^{k} q^{(i+2j+3k)(i+2j+3k-1)+3k^2+i-3j-3k}}{(q;q)_i(q^4;q^4)_j(q^6;q^6)_k}x^{i+2j+3k}
&=
\sum_{N\ge0}h_{N,-2,-2}\cdot  q^{N^2} x^N.
\end{align*}
Using \eqref{E:h9recsystem} and \eqref{E:H9Rec}, it is then not hard  to deduce that the
series above is indeed $q^{-1}(1+q+q^2)H_9(1)$. As such, this sum-product identity is equivalent to \eqref{C:H9}.

\end{document}